\documentclass[12pt,reqno]{amsart}

\usepackage[T1]{fontenc}
\usepackage[utf8]{inputenc}
\usepackage{amssymb,booktabs}
\usepackage{enumitem}
\usepackage[mathscr]{euscript}
\let\euscr\mathscr \let\mathscr\relax
\usepackage[pagebackref=false,colorlinks=true,linkcolor=red,citecolor=red]{hyperref}
\usepackage{dsfont}
\usepackage{mathtools}

\makeatletter
\def\@seccntformat#1{%
  \protect\textup{%
    \protect\@secnumfont
    \expandafter\protect\csname format#1\endcsname
    \csname the#1\endcsname
    \protect\@secnumpunct
  }
}

\makeatother

\theoremstyle{plain}
\newtheorem{Theorem}{Theorem}[section]
\newtheorem{Lemma}[Theorem]{Lemma}
\newtheorem{Proposition}[Theorem]{Proposition}
\newtheorem{Corollary}[Theorem]{Corollary}
\newtheorem{Theoremsub}{Theorem}[subsection]

\theoremstyle{definition} 
\newtheorem{Definition}[Theorem]{Definition}
\newtheorem{Remark}[Theorem]{Remark}
\newtheorem{Example}[Theorem]{Example}
\newtheorem{Examplesub}[Theoremsub]{Example}

\newcommand{\N}{{\mathbb{N}}}
\newcommand{\Q}{{\mathbb{Q}}}
\newcommand{\R}{{\mathbb{R}}}
\newcommand{\eR}{{\overline{\R}}}

\newcommand{\inff}{\textsf{\scriptsize ---}}

\newcommand{\SL}{\mathsf{S}}
\newcommand{\Cong}{\euscr{C}}

\newcommand{\powerset}{\euscr{P}(X)}
\newcommand{\Cd}{\mathscr{B}}

\newcommand{\tbigcup}{\mathop{\textstyle \bigcup }}
\newcommand{\tbigvee}{\mathop{\textstyle \bigvee }}
\newcommand{\tbigwedge}{\mathop{\textstyle \bigwedge }}
\newcommand{\tsum}{\mathop{\textstyle \sum }}

\newcommand{\subscript}[2]{$#1 _ #2$}

\newcommand{\sframes}{\mathsf{\sigma Frm}}

\newcommand{\reals}{\mathfrak{L}(\R)}
\newcommand{\ereals}{\mathfrak{L}(\overline{\R})}

\newcommand{\M}{\mathsf{M}}
\newcommand{\SM}{\mathsf{SM}}
\newcommand{\F}{\mathsf{F}}
\newcommand{\eM}{\overline{\mathsf{M}}}
\newcommand{\eF}{\overline{\mathsf{F}}}

\newcommand{\irchi}[2]{\raisebox{\depth}{$#1\chi$}}
\DeclareRobustCommand{\rchi}{{\mathpalette\irchi\relax}}

\begin{document}

\title[Lebesgue integration on {\large$\sigma$}-locales]{Lebesgue integration on {\Large$\sigma$}-locales:\\[2mm] simple functions}

\author[Raquel Bernardes]{Raquel Bernardes}

\address{\hspace*{-\parindent}CMUC, Department of Mathematics, University of Coimbra,  3000-143 Coimbra, Portugal \newline {\it Email address}: {\tt rbernardes@mat.uc.pt}}



\begin{abstract}
This paper presents a point-free version of the Lebesgue integral for simple functions on $\sigma$-locales. It describes the integral with respect to a measure defined on the coframe of all $\sigma$-sublocales, moving beyond the constraints of Boolean algebras. It also extends the notion of integrable function, usually reserved for measurable functions, to localic general functions.
\end{abstract}

\maketitle

\section{Introduction}

The concept of a ``point-free'' integral is not new in the literature. In 1965, Segal \cite{Segal1965} proposed an ``algebraic integration theory'' (i.e., concerned with features independent of isomorphisms), whose starting point is an algebraic structure with no underlying space of points.
More recently, Coquand-Palmgren \cite{CoquandPalmgren2002}, Coquand-Spitters \cite{CoquandSpitters2009} and Vickers \cite{Vickers2008} gave a constructive account of measure theory. They approached the real numbers constructively, working with lower reals, upper reals and Dedekind reals.
With a more categorical perspective, Kriz-Pultr \cite{KrizPultr2014}, Ball-Pultr \cite{BallPultr2018} and Jakl \cite{Jakl2018} proposed generalisations of measure theory to abstract $\sigma$-algebras (that is, Boolean algebras with countable joins). They took morphisms of abstract $\sigma$-algebras as measurable functions. Then, in that framework, they focused on the Daniell's version of Lebesgue integral (for an account of Daniell's integral see \cite{Pultr2010}).

The novelty of our approach in this paper
lies in our effort to move beyond the constraints of Boolean algebras and of a specified notion of ``measurability''. We aim to describe the integral with respect to measures defined on coframes and to extend the concept of integrable function (usually reserved for measurable functions) to arbitrary functions.

Just as locales generalise a substantial part of topological spaces (\cite{PicadoPultr2012}), $\sigma$-locales generalise a substantial part of  measurable spaces (\cite{BaboolalGhosh2008}). Motivated by this, Simpson  proposed in \cite{Simpson2012} a new approach to the problem of measuring subsets. Replacing ``subsets'' by ``$\sigma$-sublocales'' (which actually generalise the former when a space is viewed as a $\sigma$-locale), Simpson extended a measure $\mu$ on a $\sigma$-locale $L$ to a measure $\mu^{\diamond}$ on the coframe $\SL(L)$ of all $\sigma$-sublocales of $L$. Remarkably, this approach not only overcomes some well-known classic paradoxes (as the ones of Vitali \cite{Vitali1905} and Banach-Tarski \cite{BanachTarski1924}), but it also shows that there is no need to introduce any formal notion of measurability. Since $\mu^{\diamond}$ assigns a value to every $\sigma$-sublocale of $L$, it makes sense to assume that every $\sigma$-sublocale of $L$ is “measurable”.

In \cite{Bernardes2023}, we studied measurable functions on a $\sigma$-locale (previously mentioned by Banaschewski-Gilmour \cite{BanaschewskiGilmour2011} as $\sigma$-continuous maps). A {\em measurable function} on a $\sigma$-locale $L$ is a $\sigma$-frame homomorphism $f\colon \reals\to L$ from the frame of reals into $L$. General real-valued functions on $L$ are measurable functions on $\Cong(L)=\SL(L)^{op}$, the {\em congruence frame} of $L$.

Our aim now is  to develop a theory of integration for general real-valued functions in a $\sigma$-locale $L$  by defining the integral with respect to a measure defined on all of $\SL(L)$.
This paper is the first step on that direction: we present a point-free version of the Lebesgue integral for localic simple functions (regardless of being measurable or not).
In order to work with simple functions, we will have to introduce the limit superior and the limit inferior of a sequence of measurable functions.
This will require some investigation of the behaviour of countable joins and meets in the ring $\M(L)$ of measurable functions on a $\sigma$-locale $L$.

In a subsequent paper (\cite{Bernardes_inpreparation}) we intend to use Corollary \ref{DecompositionNonnegativeMeasFunction} (that states that under certain conditions a nonnegative function on $L$ can be written as a limit of simple functions) to extend the integral to more general functions.

\medskip
The paper is structured as follows. In Section 2, we review some essential definitions and results about $\sigma$-locales, $\sigma$-frames and the ring of measurable functions. Then, we study the existence of countable joins and countable meets of measurable functions (Section 3), in order to establish the notions of limit superior, limit inferior and limit of a sequence of measurable functions in Section 4. In Section 5, we present the ring of measurable simple functions,  and we show in Section 6 that any nonnegative measurable function on $L$ can be written as a limit of an increasing sequence of nonnegative simple functions. Then, in Section 7, we introduce the integral of simple functions and study its elementary properties. Firstly, we focus on  nonnegative simple functions; and later (Section 8), we extend it to general simple functions. In Section 9, we prove that the indefinite integral of a nonnegative simple function is a measure on $\SL(L)$. Finally, in Section 10, we show that our point-free integral generalises the classic Lebesgue integral of simple functions. We close the paper with a final comment on the integration of more general functions on $L$.

\section{Background}

Our general reference for point-free topology and lattice theory is Picado-Pultr \cite{PicadoPultr2012}. For $\sigma$-frames ($\sigma$-locales) and congruences on $\sigma$-frames, we use  Madden \cite{Madden1991} and Frith-Schauerte \cite{FrithSchauerte2018}. We follow our previous paper \cite{Bernardes2023} for measurable functions on $\sigma$-locales. And finally, for general classic measure theory and, in particular, simple functions and integrable simple functions, our main references are Halmos \cite{Halmos1950},  Evans-Gariepy \cite{EvansGariepy2015} and Bartle \cite{Bartle1995}.

\subsection{Frames and locales}

A {\em frame} is a complete lattice $L$ (with bottom 0 and top 1) satisfying the distributive law
\[
\big(\tbigvee_{a\in A}a\big)\wedge b=\tbigvee_{a\in A}(a\wedge b)
\]
for every $A\subseteq L$ and $b\in L$. Equivalently, it is a complete Heyting algebra, with the Heyting implication given by $a\to b=\bigvee\{x\in L\mid a\wedge x\le b\}$ and pseudocomplements given by $a^{*}=a\to 0=\bigvee\{x\in L\mid a\wedge x=0\}$.
Whenever the pseudocomplement of an element $a\in L$ is the complement of $a$, we shall denote it by $a^{c}$.
A {\em frame homomorphism} is a map between frames that
preserves finite meets and arbitrary joins. 

Frames and frame
homomorphisms form a category that will be denoted by \textsf{Frm}. The category of {\em locales} and {\em localic maps} is the opposite category  $\textsf{Loc}=\textsf{Frm}^{op}$.

\subsection{\texorpdfstring{$\sigma$}{[sigma]}-Frames and \texorpdfstring{$\sigma$}{[sigma]}-locales}

A lattice $L$ is {\it join-$\sigma$-complete} if it has joins of all countable $A\subseteq L$.
A join-$\sigma$-complete lattice is a {\it $\sigma$-frame} \cite{Banaschewski1980,Madden1991}  if it satisfies the distributive law
\[
\big(\tbigvee_{a\in A}a\big)\wedge b=\tbigvee_{a\in A}(a\wedge b)
\]
for every countable $A\subseteq L$ and $b\in L$. A {\em $\sigma$-frame homomorphism} is a map between $\sigma$-frames that preserves finite meets and countable joins.
The category of $\sigma$-frames and $\sigma$-frame homomorphisms will be denoted by $\sigma$\textsf{Frm}.

The category of $\sigma$-{\em locales} and $\sigma$-{\em localic maps} is the opposite category $\sigma\textsf{Loc}=\sigma\textsf{Frm}^{op}$.
A remarkable difference between \textsf{Loc} and $\sigma$\textsf{Loc} lies in the fact that while the subobjects of a locale $L$ in $\textsf{Loc}$, refered to as \emph{sublocales}, have a useful concrete description as subsets of $L$ (see \cite{PicadoPultr2012}), a similar description is not possible in $\sigma$\textsf{Loc}.
Indeed, a subobject $S$ of an object $L$ in $\sigma$\textsf{Loc}, that we refer to as a {\em $\sigma$-sublocale}, can only be described by a $\sigma$-frame quotient $L/\theta_S$ given by a {\em $\sigma$-frame congruence} $\theta_S$ on $L$, that is, an equivalence relation on $L$ satisfying the congruence properties
\begin{enumerate}
\item[(C1)] $(x,y), (x',y')\in \theta_S \ \Rightarrow \ (x\wedge x', y\wedge y')\in \theta_S$,
\item[(C2)] $\displaystyle(x_a,y_a)\in \theta_S \ (a\in A, A= \mbox{countable}) \ \Rightarrow \ \big(\tbigvee_{a\in A} x_a,\tbigvee_{a\in A} y_a\big)\in \theta_S$.
\end{enumerate}

The set $\Cong(L)$ of all congruences on a $\sigma$-frame $L$, ordered by inclusion, is a frame \cite{Madden1991}.
Hence  the dual lattice
$\SL(L)$ of all $\sigma$-sublocales of $L$, equipped with the partial order
\[
S\le T\textrm{ if and only if }\theta_T\subseteq\theta_S,
\]
is a coframe. Given a complemented $\sigma$-sublocale $S$, its complement $S^{c}$ is precisely the $\sigma$-sublocale defined by $\theta_S^{c}$.

The {\em open} and {\em closed} $\sigma$-sublocales associated with an element $a\in L$ are the $\sigma$-sublocales $\mathfrak{o}(a)$ and $\mathfrak{c}(a)$ represented, respectively, by the open and closed congruences
\begin{align*}
\Delta_a&\coloneqq\{(x,y)\in L\times L\mid x\wedge a=y\wedge a\}\\[2mm]
\nabla_a&\coloneqq\{(x,y)\in L\times L\mid x\vee a=y\vee a\}.
\end{align*}
They are complemented to each other in $\Cong(L)$.

Setting $\Delta[L]\coloneqq\{\Delta_a\mid a\in L\}$ and $\nabla[L]\coloneqq\{\nabla_a\mid  a\in L\}$, the map $\nabla\colon L\to\nabla[L]$ is an
isomorphism of $\sigma$-frames (an embedding of $L$ in $\Cong(L)$) while $\Delta\colon L^{op}\to\Delta[L]$ is an isomorphism of $\sigma$-coframes (an embedding of $L$ in $\SL(L)$, with $L$ isomorphic to $\mathfrak{o}[L]\coloneqq\{\mathfrak{o}(a)\mid a\in L\}$).

In particular,
\begin{align*}
0_{\Cong(L)}&=\nabla_{0_L}=\Delta_{1_L}=\{(x,y)\in L\times L\mid x=y\},\\
1_{\Cong(L)}&=\nabla_{1_L}=\Delta_{0_L}= L\times L,
\end{align*}
while $1_{\SL(L)}=L/0_{\Cong(L)}$ and $0_{\SL(L)}=L/1_{\Cong(L)}$ are isomorphic to $L$ and $\{0=1\}$, respectively.

\subsection{The frames of reals and of extended reals}

The {\em frame of reals} \cite{Banaschewski1997} is the frame $\reals$ generated by elements $(p,\inff)$ and $(\inff,q)$, with $p,q\in\Q$, subject to the  relations
\begin{enumerate}[label=(\subscript{R'}{{\arabic*}})]
\item $(p,\inff)\wedge(\inff,q)=0$ whenever $p\geq q$;
\item $(p,\inff)\vee(\inff,q)=1$ whenever $p<q$;
\item $(p,\inff)=\tbigvee\{(r,\inff)\mid p<r\}$;
\item $(\inff,q)=\tbigvee\{(\inff,s)\mid s<q\}$;
\item $1=\tbigvee\{(p,\inff)\mid p\in\Q\}$;
\item $1=\tbigvee\{(\inff,q)\mid q\in\Q\}$.
\end{enumerate}

The frame $\ereals$ of {\em extended reals} \cite{BanaschewskiGarciaPicado2012}  is the frame generated by all $(p,\inff)$ and $(\inff,q)$, with $p,q\in\Q$, subject to the relations
\begin{enumerate}[label=(\subscript{R'}{{\arabic*}})]
\item $(p,\inff)\wedge(\inff,q)=0$ whenever $p\geq q$;
\item $(p,\inff)\vee(\inff,q)=1$ whenever $p<q$;
\item $(p,\inff)=\tbigvee\{(r,\inff)\mid p<r\}$;
\item $(\inff,q)=\tbigvee\{(\inff,s)\mid s<q\}$.
\end{enumerate}

Since the relations involved only deal with countable joins and any countably generated $\sigma$-frame $L$ is automatically a frame, the frame $\reals$ (resp. $\ereals$) corresponds to the $\sigma$-frame defined
by the same generators and relations (see \cite{BanaschewskiGilmour2011,Simpson2012} for more details).

A map from the generating set of $\reals$ (resp. $\ereals$) into a $\sigma$-frame $L$ defines a $\sigma$-frame homomorphism $f\colon \reals\to L$ (resp. $f\colon\ereals\to L$) if and only if it sends the relations of $\reals$ (resp. $\ereals$) into identities in $L$.

\subsection{Rings of measurable functions}\label{background_MeasFunctions}

A {\em measurable real function} (\cite{Bernardes2023}) on a $\sigma$-frame $L$ is a $\sigma$-frame homomorphism $f\colon \reals\to L$ . These functions were originally introduced by Banaschewski-Gilmour \cite{BanaschewskiGilmour2011} as {\em continuous real-valued functions} for $\sigma$-frames.
Similarly, a {\em measurable extended real function} on $L$ is a $\sigma$-frame homomorphism $f\colon \ereals\to L$. 

We denote by $\M(L)$ and $\eM(L)$ the sets of all measurable real functions and all measurable extended real functions on $L$, respectively. $\M(L)$ and $\eM(L)$ are both partially ordered by
\begin{align*}
f\leq g\ &\equiv \ \forall p\in\Q,  \  f(p,\inff)\leq g(p,\inff)\\
&\Leftrightarrow \ \forall q\in\Q,  \ g(\inff,q)\leq f(\inff,q).
\end{align*}

We say that an $f\in\eM(L)$ is {\em finite} if $$\tbigvee\{f(p,\inff)\mid p\in\Q\}=1=\tbigvee\{f(\inff,q)\mid q\in\Q\}.$$  Finite measurable extended real functions are clearly in a one-to-one correspondence with  $\sigma$-frame homomorphisms $f\colon\reals\to L$. Thus, we can regard
$\M(L)$ as the subset $\{f\in\eM(L)\mid f\textrm{ is finite}\}$ of $\eM(L)$.

We will also deal with the sets
\begin{align*}
\F(L)\coloneqq\M(\Cong(L))=\sframes(\reals,\Cong(L))\\
\textrm{and }\ \ \eF(L)\coloneqq\eM(\Cong(L))=\sframes(\ereals,\Cong(L))
\end{align*}
of all arbitrary real functions and arbitrary extended real functions on a $\sigma$-frame $L$.
Identifying each $f\in\eM(L)$ with $\nabla\circ f\in\eF(L)$, we have $\eM(L)\subseteq\eF(L)$. Moreover, an $f\colon \ereals\to\Cong(L)$ is measurable on $L$ if
\[f(p,\inff),\ f(\inff,q)\in\nabla[L]\ \ \mbox{ for all }p,q\in\Q.\]

\smallskip

In \cite{Bernardes2023}, we described some of the  algebraic operations in $\eM(L)$. Here we outline the ones that we will use throughout the paper.

\begin{enumerate}
\item
Let $0<\lambda\in\Q$ and $f,g\in\eM(L)$. Then, for each $p,q\in\Q$, we have:

\begin{enumerate}
\item $(\lambda\cdot f)(p,\inff)=f(\tfrac{p}{\lambda},\inff)$ and $(\lambda\cdot f)(\inff,q)=f(\inff,\tfrac{q}{\lambda})$;
\item $-f(p,\inff)=f(\inff,-p)$ and $-f(\inff,q)=f(-q,\inff)$;
\item $(f\vee g)(p,\inff)=f(p,\inff)\vee g(p,\inff)$ and $(f\vee g)(\inff,q)=f(\inff,q)\wedge g(\inff,q)$;
\item $(f\wedge g)(p,\inff)=f(p,\inff)\wedge g(p,\inff)$ and $(f\wedge g)(\inff,q)=f(\inff,q)\vee g(\inff,q)$.
\end{enumerate}
\item
Moreover, if $f,g\in\M(L)$, then
$\displaystyle(f+g)(\inff,q)=\tbigvee_{t\in\Q}(f(\inff,t)\wedge g(\inff,q-t))$, and
$\displaystyle(f+g)(p,\inff)=\tbigvee_{t\in\Q}(f(t,\inff)\wedge g(p-t,\inff)).$

\item
If $f,g\in\M(L)$ are such that  $\pmb{0}\leq f\wedge g$, then
\begin{gather*}
(f\cdot g)(p,\inff) =
  \begin{cases}
      1 &\textrm{if }p<0\\
      \tbigvee_{s>0}f(s,\inff)\wedge g(\tfrac{p}{s},\inff) &\textrm{if }p\geq 0, \\
  \end{cases}\\
(f\cdot g)(\inff,q) =
  \begin{cases}
      0 &\textrm{if }q\leq 0 \\
      \tbigvee_{s>0}f(\inff,s)\wedge g(\inff,\tfrac{q}{s}) &\textrm{if }q>0.
  \end{cases}
\end{gather*}

\end{enumerate}

For each $r\in\Q$, we also have a nullary operation $\textbf{r}$ defined by
\[
\textbf{r}(p,\inff) =
  \begin{cases}
      1 & \textrm{if }p< r \\
      0 & \textrm{if }p\geq r
  \end{cases}
\quad\mbox{and}\quad
\textbf{r}(\inff,q) =
  \begin{cases}
      0 & \textrm{if }q\leq r \\
      1 & \textrm{if }q>r.
  \end{cases}
\]
\smallskip

An $f\in\M(L)$ is nonnegative if $f\ge \pmb{0}$. Given an $f\in\M(L)$, we define the positive part of $f$, the negative part of $f$ and the modulus of $f$, respectively, by
\[f^{+}\coloneqq f\vee\pmb{0},\qquad f^{-}\coloneqq(-f)\vee\pmb{0}\,\,\,\textrm{ and }\,\,\,\,|f|\coloneqq f^{+}+f^{-}.\]
Furthermore, for any $f\in\M(L)$, we have 
\[f=f^{+}-f^{-}.\]

\subsection{\texorpdfstring{$\sigma$}{[sigma]}-scales in \texorpdfstring{$\sigma$}{[sigma]}-frames}\label{sigma_scales}

In \cite{Bernardes2023}, we have shown that measurable functions can be produced via $\sigma$-scales as follows.

Let $\varphi\colon \Q\to L$ be a {\em $\sigma$-scale}, that is, a map for which there is a family $(c_r)_{r\in\Q}$ of elements of $L$ such that $\varphi(s)\wedge c_r=0$ whenever $s\leq r$, and $c_r\vee\varphi(s)=1$ whenever $r<s$. Then there is a  measurable extended real function $f\colon \ereals\to L$ determined by
\[
f(p,\inff)=\tbigvee_{r>p}c_r\textrm{ and }f(\inff,q)=\tbigvee_{r<q}\varphi(r)\ \ \ \textrm{for all }p,q\in\Q.
\]
If the $\sigma$-scale $\varphi$ is {\em finite}, that is, $\tbigvee\{\varphi(r)\mid r\in\Q\}=1=\tbigvee\{c_r\mid r\in\Q\}$, then the corresponding function $f$ is  finite.

\begin{Examplesub}\label{Constant_function}
\textit{Extended constant functions.}
For each $r\in\Q\cup\{\pm\infty\}$, let $\varphi_r(s)=0$ if $s\leq r$ and $\varphi_r(s)=1$ if $s>r$. The map $\varphi_r\colon \Q\to L$ is a $\sigma$-scale in $L$ that generates the measurable function $\textbf{r}\colon \ereals\to L$ given by the formulas
\[
\textbf{r}(p,\inff) =
  \begin{cases}
      1 & \textrm{if }p< r \\
      0 & \textrm{if }p\geq r
  \end{cases}
\quad\mbox{and}\quad
\textbf{r}(\inff,q) =
  \begin{cases}
      0 & \textrm{if }q\leq r \\
      1 & \textrm{if }q>r.
  \end{cases}
\]
We call this the \emph{extended constant function associated with $r$}.
If $r\in\Q$, the extended constant function $\textbf{r}$ is finite, and we simply refer to it as a {\em constant function}.
\end{Examplesub}

\begin{Examplesub}\label{Characteristic_function}
\textit{Characteristic functions.}
For each complemented $a\in L$, consider the finite $\sigma$-scale defined by $\varphi_a(r)=0$ if $r\leq 0$, $\varphi_a(r)=a^{c}$ if $0<r\leq 1$ and $\varphi_a(r)=1$ if $r>1$. The finite measurable function $$\rchi_a\colon \reals\to L$$ generated by $\varphi_a$ is given by
\[
\rchi_a(p,\inff) =
  \begin{cases}
      1 & \textrm{if }p<0\\
      a & \textrm{if }0\leq p<1 \\
      0 & \textrm{if }p\geq 1
  \end{cases}
\quad\mbox{and}\quad
  \rchi_a(\inff,q) =
  \begin{cases}
      0 & \textrm{if }q\leq 0 \\
      a^{c} & \textrm{if }0<q\leq 1 \\
      1 & \textrm{if }q>1.
  \end{cases}
\]
This is  the \emph{characteristic function associated with $a\in L$}.
\end{Examplesub}

\section{Countable joins and meets}

In the previous section we mentioned that $\eM(L)$ is closed under binary meets and binary joins. We shall now study the behaviour of countable joins and countable meets. (A similar study in $\F(L)$, for a frame $L$, was previously done in \cite{GarciaPicado2011}.)  Throughout this paper, the join and the meet of a sequence $(f_n)_{n\in\N}$ in $\eM(L)$ will be denoted by $\sup_{n\in\N}f_n$ and $\inf_{n\in\N}f_n$, respectively.
\smallskip

Let $L$ be a $\sigma$-frame and consider a sequence $(f_n\colon \ereals\to L)_{n\in\N}$ in $\eM(L)$.

\begin{Proposition}\label{Countable_meets}
If $\tbigvee_{n\in\N}f_n(\inff,q)$ is complemented in $L$ for every $q\in\Q$, then $\inf_{n\in\N}f_n$ exists in $\eM(L)$, with
\begin{align*}
(\inf_{n\in\N}f_n)(\inff,q)&=\tbigvee_{n\in\N}f_n(\inff,q),\\
(\inf_{n\in\N}f_n)(p,\inff)&=\tbigvee_{r>p}\big(\tbigvee_{n\in\N}f_n(\inff,r)\big)^{c}
\end{align*}
for all $p,q\in\Q$. Moreover, if each $f_n\in\M(L)$ and
\[\tbigvee_{q\in\Q}\big(\tbigvee_{n\in\N}f_n(\inff,q)\big)^{c}=1,\]
then $\inf_{n\in\N}f_n\in\M(L)$.
\end{Proposition}

\begin{proof}
Define $\varphi(r)\coloneqq\tbigvee_{n\in\N}f_n(\inff,r)$ for each $r\in\Q$. Set $c_r\coloneqq\varphi(r)^{c}$. Note that $\varphi\colon \Q\to L$ is an increasing map. Thus, $\varphi(r)\wedge c_s=\varphi(r)\wedge\varphi(s)^{c}=0$ for any $r\leq s$, and  $c_s\vee\varphi(r)=\varphi(s)^{c}\vee\varphi(r)=1$ otherwise. As a result, $\varphi$ is a $\sigma$-scale in $L$, and the corresponding extended real-valued function is the meet of $\{f_n\mid n\in\N\}$. From Section \ref{sigma_scales}, it is given by
\begin{align*}
(\inf_{n\in\N}f_n)(\inff,q)&=\tbigvee_{r<q}\varphi(r)=\tbigvee_{r<q}\tbigvee_{n\in\N}f_n(\inff,r)=\tbigvee_{n\in\N}\tbigvee_{r<q}f_n(\inff,r)=\tbigvee_{n\in\N}f_n(\inff,q),\\
(\inf_{n\in\N}f_n)(p,\inff)&=\tbigvee_{r>p}c_r=\tbigvee_{r>p}\big(\tbigvee_{n\in\N}f_n(\inff,r)\big)^{c}.
\end{align*}

If each $f_n$ is finite and $\tbigvee_{q\in\Q}\big(\tbigvee_{n\in\N}f_n(\inff,q)\big)^{c}=1$, we have
\begin{align*}
\tbigvee_{p\in\Q}(\inf_{n\in\N}f_n)(p,\inff)&=\tbigvee_{p\in\Q}\tbigvee_{r>p}\big(\tbigvee_{n\in\N}f_n(\inff,r)\big)^{c}=\tbigvee_{r\in\Q}\big(\tbigvee_{n\in\N}f_n(\inff,r)\big)^{c}=1\,\,\textrm{ and}\\
\tbigvee_{q\in\Q}(\inf_{n\in\N}f_n)(\inff,q)&=\tbigvee_{q\in\Q}\tbigvee_{n\in\N}f_n(\inff,q)=\tbigvee_{n\in\N}\tbigvee_{q\in\Q}f_n(\inff,q)=1.\qedhere
\end{align*}
\end{proof}

\begin{Corollary}
If $(f_n)_{n\in\N}$ is a sequence in $\M(L)$ such that $\tbigvee_{n\in\N}f_n(\inff,q)$ is complemented for every $q\in\Q$ and there is a $g\in\M(L)$ such that $g\leq f_n$ for each $n\in\N$, then $\inf_{n\in\N}f_n$ exists in $\M(L)$.
\end{Corollary}

\begin{proof}
From $f_n(\inff,r)\leq g(\inff,r)$ for all $r\in\Q$, it follows that
\[g(r,\inff)\wedge\tbigvee_{n\in\N}f_n(\inff,r)\leq\tbigvee_{n\in\N}\big(g(r,\inff)\wedge g(\inff, r)\big)=0.\]
Hence $g(r,\inff)\leq\big(\tbigvee_{n\in\N}f_n(\inff,r)\big)^{c}$ and
\[1=\tbigvee_{r\in\Q}g(r,\inff)\leq\tbigvee_{r\in\Q}\big(\tbigvee_{n\in\N}f_n(\inff,r)\big)^{c}.\qedhere\]
\end{proof}

\begin{Corollary}\label{Countable_meets_cor2}
$\eF(L)$ is closed under countable meets  in $\eM(L)$.
\end{Corollary}

Dual results regarding the existence of countable joins of measurable functions follow similarly:

\begin{Proposition}\label{Countable_joins}
If $\tbigvee_{n\in\N}f_n(p,\inff)$ is complemented in $L$ for every $p\in\Q$, then $\sup_{n\in\N}f_n$ exists in $\eM(L)$ and
\begin{align*}
(\sup_{n\in\N}f_n)(p,\inff)&=\tbigvee_{n\in\N}f_n(p,\inff),\\
(\sup_{n\in\N}f_n)(\inff,q)&=\tbigvee_{r<q}\big(\tbigvee_{n\in\N}f_n(r,\inff)\big)^{c}
\end{align*}
for all $p,q\in\Q$. Additionally, if each $f_n\in\M(L)$ and
\[\tbigvee_{p\in\Q}\big(\tbigvee_{n\in\N}f_n(p,\inff)\big)^{c}=1,\]
then $\sup_{n\in\N}f_n\in\M(L)$.
\end{Proposition}

\begin{Corollary}
If $(f_n)_{n\in\N}$ is a sequence in $\M(L)$ such that $\tbigvee_{n\in\N}f_n(p,\inff)$ is complemented for every $p\in\Q$ and there is a $g\in\M(L)$ such that $f_n\leq g$ for each $n\in\N$, then $\sup_{n\in\N}f_n$ exists in $\M(L)$.
\end{Corollary}

\begin{Corollary}\label{Countable_joins_cor2}
$\eF(L)$ is closed under countable joins  in $\eM(L)$.
\end{Corollary}

When $L$ is a frame, the countable meet given by Proposition \ref{Countable_meets} and the countable join given by Proposition \ref{Countable_joins} are defined by
\[(\inf_{n\in\N}f_n)(\inff,q)=\tbigvee_{n\in\N}f_n(\inff,q),\qquad(\inf_{n\in\N}f_n)(p,\inff)=\tbigvee_{r>p}\tbigwedge_{n\in\N}f_n(r,\inff),\]
and
\[(\sup_{n\in\N}f_n)(p,\inff)=\tbigvee_{n\in\N}f_n(p,\inff),\qquad(\sup_{n\in\N}f_n)(\inff,q)=\tbigvee_{r<q}\tbigwedge_{n\in\N}f_n(\inff,r),\]
coinciding with the formulas  in \cite{GarciaPicado2011}.

\section{Limit inferior and limit superior. Limits of functions}

We may now introduce the limit of a sequence of measurable functions.  Our motivation is the classic definition in \cite{EvansGariepy2015}.
\smallskip

Given a $\sigma$-frame $L$, let $$(f_k\colon \ereals\to L)_{k\in\N}$$ be a sequence in $\eM(L)$. The {\em limit inferior} and the {\em limit superior} of $(f_k)_{k\in\N}$ are defined as
\[
\begin{split}
\lim_{k\rightarrow+\infty}\inf f_k\coloneqq\sup_{n\geq 1}\inf_{k\geq n}f_k,\\[1mm]
\lim_{k\rightarrow+\infty}\sup f_k\coloneqq\inf_{n\geq 1}\sup_{k\geq n}f_k.
\end{split}
\]
Note that they may not exist. When they both exist and are equal, we say that the {\em limit} of $(f_k)_{k\in\N}$ exists and write
\[\lim_{k\rightarrow+\infty}f_k=\lim_{k\rightarrow+\infty}\inf f_k=\lim_{k\rightarrow+\infty}\sup f_k.\]
\smallskip

\begin{Proposition}
If the limit superior and the limit inferior of $(f_k)_{k\in\N}$ exist, then
\[
\lim_{k\rightarrow+\infty}\inf f_k\leq\lim_{k\rightarrow+\infty}\sup f_k.
\]
\end{Proposition}

\begin{proof}
Set $h_n\coloneqq\sup_{k\geq n}f_k$ and $g_m\coloneqq\inf_{k\geq m}f_k$. Then $(h_n)_{n\in\N}$ is a decreasing sequence while $(g_m)_{m\in\N}$ is increasing. Moreover, $g_k\leq h_k$ for every $k\in\N$.
Fix $n\in\N$. If $k\leq n$, then $g_k\leq g_n\leq h_n$. If $k>n$, then $g_k\leq h_k\leq h_n$. Hence
\[\sup_{k\in\N}g_k\leq h_n\]
for every $n\in\N$, which means that $\sup_{k\in\N}g_k\leq\inf_{n\in\N}h_n$, as claimed.
\end{proof}

\begin{Proposition}
Let $(f_k)_{k\in\N}$ and $(g_k)_{k\in\N}$ be sequences in $\eM(L)$.
\begin{enumerate}
\item[\em (1)] If $\lim\sup f_k$, $\lim\sup g_k$ and $\lim\sup (f_k+g_k)$ exist, then
\[
\lim_{k\rightarrow+\infty}\sup (f_k+g_k)\leq\lim_{k\rightarrow+\infty}\sup f_k+\lim_{k\rightarrow+\infty}\sup g_k.
\]
\item[\em(2)] If $\lim\inf f_k$, $\lim\inf g_k$ and $\lim\inf(f_k+g_k)$ exist, then
\[
\lim_{k\rightarrow+\infty}\inf f_k+\lim_{k\rightarrow+\infty}\inf g_k\leq\lim_{k\rightarrow+\infty}\inf(f_k+g_k).
\]
\end{enumerate}
\end{Proposition}

\begin{proof}
Note that for any $j\geq n$,
\[
\inf_{k\geq n} f_k+\inf_{k\geq n} g_k\leq f_j+g_j\leq\sup_{k\geq n} f_k+\sup_{k\geq n} g_k,
\]
which implies that
\[
\inf_{k\geq n} f_k+\inf_{k\geq n} g_k\leq \inf_{j\geq n}(f_j+g_j)\,\,\textrm{ and }\,\,\sup_{j\geq n}(f_j+g_j)\leq\sup_{k\geq n} f_k+\sup_{k\geq n} g_k.
\]

As a result, since $(\inf_{k\geq n}f_k)_{n\in\N}$ and $(\inf_{k\geq n}g_k)_{n\in\N}$ are increasing sequences, we get
\[
\sup_{n\in\N}\inf_{k\geq n} f_k+\sup_{n\in\N}\inf_{k\geq n} g_k=\sup_{n\in\N}\big(\inf_{k\geq n} f_k+\inf_{k\geq n} g_k\big)\leq\sup_{n\in\N}\inf_{j\geq n}(f_j+g_j).
\]
Similarly, since $(\sup_{k\geq n}f_k)_{n\in\N}$ and $(\sup_{k\geq n}g_k)_{n\in\N}$ are decreasing, we have
\[
\inf_{n\in\N}\sup_{k\geq n} f_k+\inf_{n\in\N}\sup_{k\geq n}g _k=\inf_{n\in\N}\big(\sup_{k\geq n} f_k+\sup_{k\geq n} g_k\big)\geq\inf_{n\in\N}\sup_{j\geq n}(f_j+g_j).\qedhere
\]
\end{proof}

\begin{Remark} Although it can be hard to check whether the limit of a sequence in $\eM(L)$ does exist or not, the answer is straightforward whenever the sequence is monotone:

If $(f_k)_{k\in\N}$ is an increasing (resp. decreasing) sequence, its limit exists if and only if $\sup_{k\in\N}f_k$ (resp. $\inf_{k\in\N}f_k$) exists, and
\[\lim_{k\rightarrow+\infty}f_k=\sup_{k\in\N}f_k\qquad\textrm{(resp. }\lim_{k\rightarrow+\infty}f_k=\inf_{k\in\N}f_k).
\]
\end{Remark}

Combining the previous remark with Corollaries \ref{Countable_meets_cor2} and \ref{Countable_joins_cor2}, one ensures the existence of the limit of a monotone sequence of measurable functions in $\eF(L)$. All that is left to check is whether the limit is measurable or not.

\begin{Proposition}
Every monotone sequence in $\eM(L)$ has a limit in $\eF(L)$. Every monotone and bounded sequence in $\M(L)$ has a limit in $\F(L)$.
\end{Proposition}

In the following example, we describe a simplified way of determining the limit of a monotone sequence.

\begin{Example}
Let us compute the limit of the decreasing sequence $(f_n)_{n\in\N}$ given by $f_n=\pmb{\tfrac{1}{n}}$ (recall Example \ref{Constant_function}). Since it is a bounded and monotone sequence of finite measurable functions, from the previous proposition we know that its limit exists and
\[
\lim_n f_n=\inf_{n\in\N}f_n\in\F(L).
\]
Moreover, for each $q\in\Q$,
\[(\inf_{n\in\N}f_n)(\inff,q)=\tbigvee_{n\in\N}f_n(\inff,q)=
\begin{cases}
0 &\textrm{if }q\leq0\\
1 &\textrm{if }q>1.
\end{cases}
\]
Indeed, if $q\leq 0$, $f_n(\inff,q)=0$ for all $n\in\N$.
Otherwise, there is $n_q\in\N$ such that $0<\tfrac{1}{n_q}<q$ and
\[\tbigvee_{n\in\N}f_n(\inff,q)\geq f_{n_q}(\inff,q)=1.\]

Because $\varphi(r)\coloneqq(\inf_{n\in\N}f_n)(\inff,r)$ is a $\sigma$-scale in $\Cong(L)$ generating $\inf_{n\in\N}f_n$ (see \cite[Proposition 3.8]{Bernardes2023} for more details) and $\Cong(L)$ is a frame, for each $p\in\Q$,
\[(\inf_{n\in\N}f_n)(p,\inff)=\tbigvee_{r>p}(\inf_{n\in\N}f_n)(\inff,r)^{c}=
\begin{cases}
1 &\textrm{if }p<0\\
0 &\textrm{if }p\geq1.
\end{cases}
\]
Hence, both $(\inf_{n\in\N}f_n)(p,\inff)$ and $(\inf_{n\in\N}f_n)(\inff,q)$ are in $\nabla[L]$ for all $p,q\in\Q$, which means that  $\inf_{n\in\N}f_n$ is measurable on $L$. More precisely, we have
\[\lim_n\pmb{\tfrac{1}{n}}=\pmb{0}.\]
\end{Example}

\section{Simple functions}

In this section, we study a special type of $\sigma$-localic functions, the measurable simple functions, which are the point-free counterparts of finite linear combinations of characteristic functions associated with measurable sets. This class of functions forms a subring of $\M(L)$ and will be crucial to introduce integration in $\sigma$-frames.
\smallskip

Let $L$ be a $\sigma$-frame. We will denote by $\Cd L\coloneqq\{a\in L\mid a\textrm{ is complemented}\}$ the sublattice of complemented elements in $L$.

\begin{Definition}\label{Def_Simple_function}
We say that a measurable function $f\colon\reals\to L$ is {\em simple} when it is a finite linear combination of characteristic functions with rational scalars, that is, if there exist $n\in\N$, $r_1,\ldots ,r_n\in\Q$ and $a_1,\ldots ,a_n\in\Cd L$ such that
\[f=\tsum_{i=1}^{n}r_i\cdot\rchi_{a_i}.\]
Whenever $r_1<r_2<\cdots<r_n$ and $a_1,\ldots ,a_n\in\Cd L\smallsetminus\{0\}$ are pairwise disjoint with $\tbigvee_{i=1}^{n}a_i=1$, we say that $\tsum_{i=1}^{n}r_i\cdot\rchi_{a_i}$ is the {\em canonical form} of $f$.
\end{Definition}

In particular, characteristic functions are simple. Recalling Example \ref{Characteristic_function} and the algebraic operations in $\M(L)$ presented in Section \ref{background_MeasFunctions}, it is a straightforward exercise to prove the next lemma. It summarises some basic properties that will be helpful throughout this section.

\begin{Lemma}\label{Characteristic_prop}
The following statements hold for a $\sigma$-frame $L$:
\begin{enumerate}[label={\em(\arabic*)}]
\item $\rchi_{0_L}=\pmb{0}$. Moreover, for any $a,b\in\Cd L$, $a\leq b$ if and only if $\rchi_a\leq\rchi_b$.
\item For any $a,b\in \Cd L$, $\rchi_a\cdot\rchi_b=\rchi_{a\wedge b}$.
\item If $a_1,\ldots ,a_k\in \Cd L$ are pairwise disjoint, then $\rchi_{a_1}+\cdots+\rchi_{a_k}=\rchi_{\vee_{i=1}^{k}a_i}$.
\end{enumerate}
\end{Lemma}

From any representation of a simple function, we can always obtain its canonical form.

\begin{Proposition}\label{FSimple_equiv}
The following statements are equivalent for a $\sigma$-frame $L$ and an $f\in\M(L)$:
\begin{enumerate}[label={\em(\arabic*)}]
\item There are $r_1,\ldots,r_n\in\Q$ and $a_1,\ldots,a_n\in\Cd L$ such that $f=\tsum_{i=1}^{n}r_i\cdot\rchi_{a_i}$.
\item There are $r_1,\ldots,r_n\in\Q$ and pairwise disjoint $a_1,\ldots,a_n\in\Cd L\smallsetminus\{0\}$  such that $f=\tsum_{i=1}^{n}r_i\cdot\rchi_{a_i}$.
\item There are $r_1,\ldots,r_n\in\Q$ and pairwise disjoint $a_1,\ldots,a_n\in\Cd L\smallsetminus\{0\}$  with $\tbigvee_{i=1}^{n}a_i=1$ such that $f=\tsum_{i=1}^{n}r_i\cdot\rchi_{a_i}$.
\item There are $r_1<\cdots<r_n\in\Q$ and pairwise disjoint $a_1,\ldots,a_n\in\Cd L\smallsetminus\{0\}$  with $\tbigvee_{i=1}^{n}a_i=1$ such that $f=\tsum_{i=1}^{n}r_i\cdot\rchi_{a_i}$.
\end{enumerate}
\end{Proposition}

\begin{proof}
(1)$\Rightarrow$(2): Since $a_i=0_L$ implies that $\rchi_{a_i}=\pmb{0}$, we may assume without loss of generality that $a_1,\ldots ,a_n\in\Cd L\smallsetminus\{0\}$. Moreover, as $a_i,a_j\in\Cd L$, we have $a_i=(a_i\wedge a_j)\vee(a_i\wedge a_j^{c})$ and $a_j=(a_j\wedge a_i)\vee(a_j\wedge a_i^{c})$. Therefore,
\[
\begin{split}
r_i\cdot\rchi_{a_i}+r_j\cdot\rchi_{a_j}&=r_i\cdot\rchi_{a_i\wedge a_j}+r_i\cdot\rchi_{a_i\wedge a_j^{c}}+r_j\cdot\rchi_{a_j\wedge a_i}+r_j\cdot\rchi_{a_j\wedge a_i^{c}}\\
&=(r_i+r_j)\cdot\rchi_{a_i\wedge a_j}+r_i\cdot\rchi_{a_i\wedge a_j^{c}}+r_j\cdot\rchi_{a_j\wedge a_i^{c}},
\end{split}
\]
where $a_i\wedge a_j$, $a_i\wedge a_j^{c}$ and $a_j\wedge a_i^{c}$ are pairwise disjoint.

\smallskip
\noindent
(2)$\Rightarrow$(3): If $\tbigvee_{i=1}^{n}a_i\neq 1$, take $a\coloneqq(\tbigvee_{i=1}^{n}a_i)^{c}$. Then $a\neq 0$ and
\[f=\tsum_{i=1}^{n}r_i\cdot\rchi_{a_i}=\tsum_{i=1}^{n}r_i\cdot\rchi_{a_i}+0\cdot\rchi_{a},\]
where $a_1,\ldots ,a_n,a\in\Cd L\smallsetminus\{0\}$ are pairwise disjoint because $a_i\wedge a\leq a^{c}\wedge a=0$ for every $i$.

\smallskip
\noindent
(3)$\Rightarrow$(4): If $r_i=r_j$, it follows from $a_i\wedge a_j=0$ that
\[
r_i\cdot\rchi_{a_i}+r_j\cdot\rchi_{a_j}=r_i\cdot(\rchi_{a_i}+\rchi_{a_j})=r_i\cdot\rchi_{a_i\vee a_j}.
\]
In addition, for $k\in\{1,\ldots ,n\}\smallsetminus\{i,j\}$, $(a_i\vee a_j)\wedge a_k=(a_i\wedge a_k)\vee(a_j\wedge a_k)=0$.

\smallskip
\noindent
$(4)\Rightarrow(1)$ is immediate.
\end{proof}

\begin{Proposition}\label{Other_formula}
Let  $f=\tsum_{i=1}^{n}r_{i}\cdot\rchi_{a_i}$ be a measurable simple function on $L$ with $r_1<\cdots<r_k<0\leq r_{k+1}<\cdots<r_n\in\Q$ and pairwise disjoint $a_1,\ldots,a_n\in\Cd L$. For each $p,q\in\Q$, $f(p,\inff)$ and $f(\inff,q)$ are given by the following formulas:

\medskip
\begin{center}
\begin{tabular}{cc|cc}
\toprule[1.25pt] \addlinespace[0.5em]
\multicolumn{2}{c}{$f(p,\inff)$} &  \multicolumn{2}{c}{$f(\inff,q)$}\\[2mm]
\toprule[1.25pt] \addlinespace[0.5em]
      $1$ & $\textrm{ if }p<r_1$ & $0$ & $\textrm{ if }q\leq r_1$\\[2mm]
      $a_1^{c}$ & $\textrm{ if }r_1\leq p<r_2$ & $a_1$ & $\textrm{ if }r_1<q\leq r_2$\\[2mm]
            $(a_1\vee a_2)^{c}$ & $\textrm{ if }r_2\leq p<r_3$ & $a_1\vee a_2$ & $\textrm{ if }r_2<q\leq r_3$\\[2mm]
      $\vdots$ & $\vdots$ & $\vdots$ & $\vdots$\\[2mm]
      $\displaystyle\big(\tbigvee_{i=1}^{k-1}a_i\big)^{c}$ &$\textrm{ if }r_{k-1}\leq p<r_{k}$ & $\displaystyle\tbigvee_{i=1}^{k-1}a_i$ &$\textrm{ if }r_{k-1}<q\leq r_{k}$\\[5mm]
      $\displaystyle\tbigvee_{i=k+1}^{n}a_i$ &$\textrm{ if }r_k\leq p<r_{k+1}$ & $\displaystyle\big(\tbigvee_{i=k+1}^{n}a_i\big)^{c}$ &$\textrm{ if }r_k<q\leq r_{k+1}$\\[2mm]
$\vdots$ & $\vdots$ & $\vdots$ & $\vdots$\\[2mm]
$a_{n-1}\vee a_n$ & $\textrm{ if }r_{n-2}\leq p<r_{n-1}$ & $(a_{n-1}\vee a_n)^{c}$ & $\textrm{ if }r_{n-2}< p\leq r_{n-1}$\\[2mm]
$a_n$ & $\textrm{ if }r_{n-1}\leq p<r_{n}$ & $a_n^{c}$ & $\textrm{ if }r_{n-1}< p\leq r_{n}$\\[2mm]
      $0$ & $\textrm{ if }r_n\leq p$ &
      $1$ &$\textrm{ if }r_n<q$\\[2mm]
      \bottomrule[1.25pt]
\end{tabular}
\end{center}
\end{Proposition}

\medskip
\begin{proof}
Since the proof of this result is lengthy but straightforward, we will only outline each of its three parts.
Firstly, we suppose that $0<r_1<\cdots<r_n$ and prove the claim by induction over $n\in\N$ (applying the formulas for sum in $\M(L)$). Next, we apply the previous step to deduce the cases when (I) $0\leq r_1<\cdots<r_n$ and (II) $r_1<\cdots<r_n<0$. Finally, if $r_1<\cdots<r_k<0\leq r_{k+1}<\cdots<r_n$, note that $f=f_k+f_{n-k}$, where
\[
f_k\coloneqq\tsum_{i=1}^{k}r_{i}\cdot\rchi_{a_i}\qquad\textrm{ and }\qquad f_{n-k}\coloneqq\tsum_{i=k+1}^{n}r_{i}\cdot\rchi_{a_i}.
\]
Thus, we can use (I) and (II) to obtain the formulas for $f_{n-k}$ and $f_k$, respectively. Then the result follows straightforwardly from the sum in $\M(L)$.
\end{proof}

In particular, if $a_1,\ldots ,a_n\in\Cd L$ are pairwise disjoint and $\tbigvee_{i=1}^{n}a_i=1$, we have $(\tbigvee_{i=1}^{j}a_i)^{c}=\tbigvee_{i=j+1}^{n}a_i$ for any $j\in\{1,\ldots ,n-1\}$. Hence, we get the following formulas describing the canonical representation of a simple function.

\begin{Corollary}\label{SFunction_canonical}
Let $f=\tsum_{i=1}^{n}r_{i}\cdot\rchi_{a_i}$ be a measurable simple function on $L$ with $r_1<\cdots <r_n\in\Q$ and pairwise disjoint $a_1,\ldots ,a_n\in\Cd L$ such that $\tbigvee_{i=1}^{n}a_i=1$. For each $p,q\in\Q$, $f(p,\inff)$ and $f(\inff,q)$ are given by the following formulas:

\medskip
\begin{center}
\begin{tabular}{cc|cc}
\toprule[1.25pt] \addlinespace[0.5em]
\multicolumn{2}{c}{$f(p,\inff)$} &  \multicolumn{2}{c}{$f(\inff,q)$}\\[2mm]
\toprule[1.25pt]\addlinespace[0.5em]
$1$ & if $r_1<p$ & $0$ & if $q\leq r_1$\\[3mm]
      $\displaystyle\tbigvee_{i=2}^{n}a_i$ &if $r_1\leq p<r_2$ &  $a_1$ & if $r_1<q\leq r_2$\\[5mm]
      $\displaystyle\tbigvee_{i=3}^{n}a_i$ &  if $r_2\leq p<r_3$ &     $a_1\vee a_2$ &  if $r_2<q\leq r_3$\\[3mm]
      $\vdots$ & $\vdots$ &
      $\vdots$ & $\vdots$\\[3mm]
      $a_n$ &  if $r_{n-1}\leq p<r_n$ &  $\displaystyle\tbigvee_{i=1}^{n-1}a_i$ &  if $r_{n-1}<q\leq r_n$\\[4mm]
      $0$ &  if $r_n\leq p $ &       $1$ & if $r_n<q$ \\[2mm]
\bottomrule[1.25pt]
\end{tabular}
\end{center}
\end{Corollary}

\medskip
\begin{Remark}
It follows from the above corollary that each simple function $f\colon\reals\to L$ is uniquely determined by its canonical form. Indeed, if $\tsum_{i=1}^{n}r_{i}\cdot\rchi_{a_i}$ and $\tsum_{i=1}^{m}s_{i}\cdot\rchi_{b_i}$ are canonical forms of $f$, then $n=m$ and $r_j=s_j$ for all $1\leq j\leq n$. Besides,
\begin{gather*}
a_1=f(\inff,r_2)=b_1,\\
a_j=\big(\tbigvee_{i=1}^{j-1}a_i\vee\tbigvee_{i=j+1}^{n}a_i\big)^{c}=\big(f(\inff,r_{j})\vee f(r_{j},\inff)\big)^{c}=\big(\tbigvee_{i=1}^{j-1}b_i\vee\tbigvee_{i=j+1}^{n}b_i\big)^{c}=b_j,\\
\textrm{ and }a_n=f(r_{n-1},\inff)=b_n.
\end{gather*}
\end{Remark}

Let
\[\SM(L)\coloneqq\{f\in\M(L)\mid f\textrm{ is simple}\}.\]

\begin{Proposition}\label{Prop_SimpleFunctions}
If $f,g\in\SM(L)$ and $\lambda\in\Q$, then $\lambda\cdot f$, $-f$, $f\cdot g$, $f+g$ and $|f|$ are simple functions. Consequently, $\SM(L)$ is a subring of $\M(L)$.
\end{Proposition}

\begin{proof}
Suppose that $f=\tsum_{i=1}^{n}r_i\cdot\rchi_{a_i}$ and $g=\tsum_{j=1}^{m}s_j\cdot\rchi_{b_j}$ are the canonical representations of $f$ and $g$, respectively.
As a consequence of $\M(L)$ being an algebra over $\Q$, we have that
\[
\lambda\cdot f=\tsum_{i=1}^{n}\lambda\cdot(r_i\cdot\rchi_{a_i})=\tsum_{i=1}^{n}(\lambda r_i)\cdot\rchi_{a_i}
\]
and, in particular, for $\lambda=-1$,
\[
-f=-\big(\tsum_{i=1}^{n}r_i\cdot\rchi_{a_i}\big)=\tsum_{i=1}^{n}(-r_i)\cdot\rchi_{a_i}.
\]

Moreover, by Lemma \ref{Characteristic_prop}, we have
\[
f\cdot g=\tsum_{i=1}^{n}\tsum_{j=1}^{m}(r_is_j)\cdot(\rchi_{a_i}\cdot\rchi_{b_j})=\tsum_{i=1}^{n}\tsum_{j=1}^{m}(r_is_j)\cdot\rchi_{a_i\wedge b_j}.
\]

As $\pmb{1}=\rchi_{1_L}=\rchi_{\vee_{j=1}^{m}b_j}=\sum_{j=1}^{m}\rchi_{b_j}$, for each $i\in\{1,\ldots ,n\}$,
\[
r_i\cdot\rchi_{a_i}=r_i\cdot\rchi_{a_i}\cdot\big(\tsum_{j=1}^{m}\rchi_{b_j}\big)=\tsum_{j=1}^{m}r_i\cdot\rchi_{a_i\wedge b_j}.
\]
Similarly, $s_j\cdot\rchi_{b_j}=\tsum_{i=1}^{n}s_j\cdot\rchi_{a_i\wedge b_j}$ for each $j\in\{1,\ldots ,m\}$.
Therefore,
\[
\begin{split}
f+g&=\tsum_{i=1}^{n}r_i\cdot\rchi_{a_i}+\tsum_{j=1}^{m}s_j\cdot\rchi_{b_j}=\tsum_{i=1}^{n}\tsum_{j=1}^{m}r_i\cdot\rchi_{a_i\wedge b_j}+\tsum_{j=1}^{m}\tsum_{i=1}^{n}s_j\cdot\rchi_{a_i\wedge b_j}\\
&=\tsum_{i=1}^{n}\tsum_{j=1}^{m}(r_i+s_j)\cdot\rchi_{a_i\wedge b_j}.
\end{split}
\]

Finally, suppose with no loss of generality  that $r_1<\cdots < r_k<0\leq r_{k+1}<\cdots< r_n$. It is easy to check that
\[
f^{-}=-\big(\tsum_{i=1}^{k}r_i\cdot\rchi_{a_i}\big)\,\,\,\textrm{ and }\,\,\, f^{+}=\tsum_{i=k+1}^{n}r_i\cdot\rchi_{a_i}.
\]
Hence, $|f|=f^{+}+f^{-}$ is also a simple function.
\end{proof}

\begin{Remark}
For each $\sigma$-frame $L$, we have
\[\SM(L)\subseteq\M(L)\cap\SM(\Cong(L))\subseteq\F(L)\subseteq\eF(L).\]
\end{Remark}

\section{Decomposition of nonnegative measurable
functions}\label{DecompositionNonnegativeFunctions_Section}

A well-known result in measure theory states that any nonnegative measurable function is a pointwise limit of an increasing sequence of nonnegative simple functions (\cite[Theorem 1.12]{EvansGariepy2015}). In this section, we establish a counterpart of this result for point-free measurable functions. This is one of the most important results about the class of measurable simple functions. For that we need the following technical lemma.

\begin{Lemma}\label{DecompositionNonnegativeFunction_aux}
Let $f\colon \ereals\to L$ be a measurable function on $L$ such that $f(r,\inff)\in\Cd L$ for all $r\in\Q$.
Consider the sequence $(f_k\colon \ereals\to L)_{k\in\N}$ in $\SM(L)$ defined by
\begin{align*}
&a_1\coloneqq f(1,\inff),\,\,\,f_1\coloneqq\rchi_{a_1}\\
k\geq 2\colon \,\,&a_k\coloneqq(f-f_{k-1})(\tfrac{1}{k},\inff),\,\,\,f_k\coloneqq\tsum_{i=1}^{k}\tfrac{1}{i}\cdot\rchi_{a_i}.
\end{align*}
For each $k\in\N$, $f_k(p,\inff)$ and $f_k(\inff,q)$ $(p,q\in\Q)$ are given by

\medskip
\begin{center}
\begin{tabular}{cc|cc}
\toprule[1.25pt] \addlinespace[0.5em]
\multicolumn{2}{c}{$f_k(p,\inff)$} &  \multicolumn{2}{c}{$f_k(\inff,q)$}\\[2mm]
\toprule[1.25pt]\addlinespace[0.5em]
$1$ & if $p< t_0$ & $0$ &if $q\leq t_0$\\[2mm]
$f(t_1,\inff)$ & if $t_0\leq p< t_1$ & $f(t_1,\inff)^{c}$ & if $t_0<q\leq t_1$\\[2mm]
$f(t_2,\inff)$ & if $t_1\leq p<t_2$  & $f(t_2,\inff)^{c}$ & if $t_1<q\leq t_2$\\[2mm]
$\vdots$ & $\vdots$ & $\vdots$ &$\vdots$\\[2mm]
$f(t_j,\inff)$ & if $t_{j-1}\leq p< t_j$ &
$f(t_j,\inff)^{c}$ & if $t_{j-1}<q\leq t_j$\\[2mm]
$\vdots$ &$\vdots$ & $\vdots$ &$\vdots$\\[2mm]
$f(t_n,\inff)$ & if $t_{n-1}\leq p< t_n $ & $f(t_n,\inff)^{c}$ & if $t_{n-1}<q\leq t_n$\\[2mm]
$0$ &$\textrm{if }p\geq t_n$
& $1$ & if $q>t_n$ \\[2mm]
\bottomrule[1.25pt]
\end{tabular}
\end{center}

\medskip\noindent
for some $n\in\N$, $t_0, t_1,\ldots ,t_n\in\Q$ satisfying
\[
t_0\coloneqq0<t_1\coloneqq\tfrac{1}{k}<t_2<\cdots<t_{n-2}<t_{n-1}\coloneqq\tsum_{i=1}^{k-1}\tfrac{1}{i}<t_n\coloneqq\tsum_{i=1}^{k}\tfrac{1}{i}
\]
and $t_j-t_{j-1}\leq\tfrac{1}{k}$ for every $j\in\{1,\ldots ,n\}$.
\end{Lemma}

\begin{proof}
We shall verify that the sequence $(f_k)_{k\in\N}$ is well-defined and given by the formulas above. Moreover, since $f_k(p,\inff)=\tbigvee_{r>p}f_k(r,\inff)=\tbigvee_{r>p}f_k(\inff,r)^{*}$, we just need to prove the formula for $f_k(\inff,q)$ by induction over $k\in\N$.

As $a_1\in\Cd L$, it follows from Example \ref{Characteristic_function} that the claim holds for $k=1$.
Now, suppose that the claim is true for $k-1\geq 1$. In other words, suppose that $a_1,\ldots,a_{k-1}\in\Cd L$ and
\[
f_{k-1}(\inff,q)=
\begin{cases}
0 &\textrm{if }q\leq 0\eqqcolon t_0\\[2mm]
f(t_1,\inff)^{c} &\textrm{if }t_0<q\leq\tfrac{1}{k-1}\eqqcolon t_1\\[2mm]
f(t_2,\inff)^{c} &\textrm{if }t_1<q\leq t_2\\[2mm]
\vdots &\,\,\,\vdots\\[2mm]
f(t_j,\inff)^{c} &\textrm{if }t_{j-1}<q\leq t_j\\[2mm]
\vdots &\,\,\,\vdots\\[2mm]
f(t_n,\inff)^{c} &\textrm{if }t_{n-1}<q\leq\tsum_{i=1}^{k-1}\tfrac{1}{i}\eqqcolon t_n\\[2mm]
1 &\textrm{if }q>t_n,
\end{cases}
\]
with $t_j-t_{j-1}\leq\tfrac{1}{k-1}$ for every $j\in\{1,\ldots ,n\}$.
Then, for each $k\in\N$,
\[
a_k\coloneqq(f-f_{k-1})(\tfrac{1}{k},\inff)=\tbigvee_{i=1}^{n}\big(f(t_i,\inff)^{c}\wedge f(t_{i-1}+\tfrac{1}{k},\inff)\big)\vee f(t_n+\tfrac{1}{k},\inff)
\]
is a complemented element in $L$ (since $\Cd L$ is a sublattice of $L$) and we can define
\[f_k\coloneqq\tsum_{i=1}^{k}\tfrac{1}{i}\cdot\rchi_{a_i}=f_{k-1}+\tfrac{1}{k}\cdot\rchi_{a_k}.\]

Note that
\[
\rchi_{a_k}(\inff, k(q-r))=
\begin{cases}
1 &\textrm{if }  r< q-\tfrac{1}{k}\\
a_k^{c} &\textrm{if } q-\tfrac{1}{k}\leq r< q\\
0 &\textrm{if }r\geq q.
\end{cases}
\]
Hence, by the algebraic operations described in Section \ref{background_MeasFunctions},
\begin{align*}
f_k(\inff,q)&=\tbigvee_{r\in\Q}f_{k-1}(\inff,r)\wedge\rchi_{a_k}(\inff,k(q-r))\\
&=\tbigvee_{0<r<q}f_{k-1}(\inff,r)\wedge\rchi_{a_k}(\inff,k(q-r)).
\end{align*}

\noindent
I. For $q\leq 0$, it is clear that $f_k(\inff,q)=0$.

\noindent
II. Suppose that $t_0=0<q\leq t_1=\tfrac{1}{k-1}$.
\smallskip

(i) If $0<q\leq\tfrac{1}{k}$, we have $q-\tfrac{1}{k}\leq0<q$. So
\[f_k(\inff,q)=f(t_1,\inff)^{c}\wedge a_k^{c}=(f(t_1,\inff)\vee a_k)^{c}=f(\tfrac{1}{k},\inff)^{c}\]
in view of the fact that
\[f(t_1,\inff)\vee a_k=\tbigvee_{i=1}^{n}\big(f(t_1,\inff)\vee f(t_{i-1}+\tfrac{1}{k},\inff)\big)=f(t_1,\inff)\vee f(t_0+\tfrac{1}{k},\inff).\]

(ii) Otherwise, when $\tfrac{1}{k}<q\leq t_1$, we get $0<q-\tfrac{1}{k}<q\leq t_1$ and
\[f_k(\inff,q)=(f(t_1,\inff)^{c}\wedge a_k^{c})\vee(f(t_1,\inff)^{c}\wedge 1)=f(t_1,\inff)^{c}.\]

\noindent
III. Now fix $j=2,\ldots ,n$ and consider $t_{j-1}<q\leq t_j$. As $0=t_0<q-\tfrac{1}{k}$, there is $m\coloneqq\max\{i\in\{0,\ldots ,n\}\mid t_i<q-\tfrac{1}{k}\}.$
\smallskip

(i) If $t_{j-1}<q\leq (t_{j-1}+\tfrac{1}{k})\wedge t_j$,
note that $m\leq j-2$ and $q-\tfrac{1}{k}\leq t_{m+1}\wedge t_{j-1}$. Thus,
\[
\begin{split}
f_k(\inff,q)&=f(t_{m+1},\inff)^{c}\vee(f(t_j,\inff)^{c}\wedge a_k^{c})\\
&=[f(t_{m+1},\inff)\wedge(f(t_j,\inff)\vee a_k)]^{c}=f(t_j\wedge(t_{j-1}+\tfrac{1}{k}),\inff)^{c}.
\end{split}
\]

The proof of the last equality is rather technical. Succinctly we have
\[
\begin{split}
f(t_j,\inff)\vee a_k&=\tbigvee_{i=1}^{n}\Big([f(t_j,\inff)\vee f(t_i,\inff)^{c}]\wedge[f(t_j,\inff)\vee f(t_{i-1}+\tfrac{1}{k},\inff)]\Big)\\
&=\tbigvee_{i=1}^{n}\Big([f(t_j,\inff)\vee f(t_i,\inff)^{c}]\wedge f(t_j\wedge(t_{i-1}+\tfrac{1}{k}),\inff) \Big)\\
&=f(t_j\wedge(t_{j-1}+\tfrac{1}{k}),\inff)\vee\tbigvee_{i=1}^{j-1}\big(f(t_i,\inff)^{c}\wedge f(t_j\wedge(t_{i-1}+\tfrac{1}{k}),\inff)\big).
\end{split}
\]
Since $t_{m+1}\leq t_j\wedge(t_{j-1}+\tfrac{1}{k})$, $f(t_j\wedge(t_{j-1}+\tfrac{1}{k}),\inff)\leq f(t_{m+1},\inff)$. And
\[
f(t_{m+1},\inff)\wedge\tbigvee_{i=1}^{j-1}\big(f(t_i,\inff)^{c}\wedge f(t_j\wedge(t_{i-1}+\tfrac{1}{k}),\inff)\big)=0\\
\]
in consequence of $f(t_j\wedge(t_{i-1}+\tfrac{1}{k}),\inff)\leq f(t_{j-1}\wedge(t_{i-1}+\tfrac{1}{k}),\inff)$,
$f(t_{m+1},\inff)\wedge f(t_i,\inff)^{c}=0$ for $0\leq i\leq m+1$,
and $t_{j-1}< q\leq t_{m+1}+\tfrac{1}{k}\leq t_{i-1}+\tfrac{1}{k}$ for $m+2\leq i\leq j-1$.
\smallskip

(ii) If $t_{j-1}+\tfrac{1}{k}< t_j$, take $t_{j-1}+\tfrac{1}{k}<q\leq t_j$. Then $0<t_1<\cdots<t_{j-1}<q-\tfrac{1}{k}<q$ and
\[
f_k(\inff,q)=f(t_j,\inff)^{c}.
\]

\noindent
IV. When $t_n<q\leq t_n+\tfrac{1}{k}$, setting $m\coloneqq\max\{i\in\{0,\ldots ,n\}\mid t_i<q-\tfrac{1}{k}\}$, we have $0<t_1<\cdots<t_m<q-\tfrac{1}{k}\leq t_{m+1}\leq t_n<q$ and
\[
\begin{split}
f_k(\inff,q)=f(t_{m+1},\inff)^{c}\vee(f(t_n,\inff)^{c}\wedge a_k^{c})\vee a_k^{c}&=[f(t_{m+1},\inff)\wedge a_k]^{c}\\
&=f(t_n+\tfrac{1}{k},\inff)^{c}.
\end{split}
\]
The last equality follows from $f(t_{m+1},\inff)\wedge f(t_i,\inff)^{c}=0$ for $0\leq i\leq m+1$, and $t_n< q\leq t_{m+1}+\tfrac{1}{k}\leq t_{i-1}+\tfrac{1}{k}$ for $m+2\leq i\leq n$.
\smallskip

\noindent
V. Finally, for $q>t_n+\tfrac{1}{k}$, we have $q-\tfrac{1}{k}>t_n$ and there is $s\in\Q$ such that $q-\tfrac{1}{k}>s>t_n$. Therefore,
\[f_k(\inff,q)\geq f_{k-1}(\inff,s)\wedge\rchi_{a_k}(\inff,k(q-s))=1.\]

This concludes the proof that $f_k(\inff,q)$ satisfies the required formula. The remaining facts are straightforward.
\end{proof}

With the sequence given above, we can tackle the problem of approximating a nonnegative measurable function by simple functions.

\begin{Theorem}\label{DecompositionNonnegativeMeasFunction}
Let $f\colon \ereals\to L$ be a nonnegative measurable function such that $f(r,\inff)\in\Cd L$ for all $r\in\Q$. If any countable join in $\{f(r,\inff)\mid r\in\Q\}$ is complemented in $L$, then we can write
\[
f=\lim_{k\rightarrow+\infty}\tsum_{i=1}^{k}\tfrac{1}{i}\cdot\rchi_{a_i}
\]
for some $a_i\in\Cd L$.
\end{Theorem}

\begin{proof}
Consider the sequence $(f_k)_{k\in\N}$ in $\SM(L)$ given by
\begin{align*}
&a_1\coloneqq f(1,\inff),\,\,\,f_1\coloneqq\rchi_{a_1}\\
k\geq 2:\,&a_k\coloneqq(f-f_{k-1})(\tfrac{1}{k},\inff),\,\,\,f_k\coloneqq\tsum_{i=1}^{k}\tfrac{1}{i}\cdot\rchi_{a_i},
\end{align*}
(already used in Lemma \ref{DecompositionNonnegativeFunction_aux}).
Each $a_k$ is a finite join of elements of the form $a^{c}\wedge b$ for some $a,b\in\Cd L$. The sequence $(f_k)_{k\in\N}$ is increasing. And because $f$ is nonnegative, we have $f_k\leq f$ for each $k\in\N$.
Moreover, for every $p\in\Q$, $\tbigvee_{k\in\N}f_k(p,\inff)=\tbigvee_{k\in\N}f(p_k,\inff)$ for some $p<p_k\in\Q$ ($k\in\N$). Hence, since countable joins of $\{f(r,\inff)\mid r\in\Q\}$ are complemented in $L$, $\tbigvee_{k\in\N}f_k(p,\inff)$ is complemented in $L$. As a consequence, $\sup_{k\in\N}f_k$ exists in $\eM(L)$ and
\[
\lim_{k\rightarrow+\infty} f_k=\sup_{k\in\N}f_k.
\]
We want to show that
\[f=\lim_{k\rightarrow+\infty} f_k.\]

Because $f_k\leq f$ for each $k\in\N$, it is clear that
\[\lim_{k\rightarrow+\infty} f_k=\sup_{k\in\N}f_k\leq f.\]

Conversely, if $p<0$, as $f_k(p,\inff)=1$ for each $k\in\N$ and $f\geq\pmb{0}$ implies that $f(p,\inff)=1$, then
\[
(\lim_{k\rightarrow+\infty} f_k)(p,\inff)=\tbigvee_{k\in\N}f_k(p,\inff)=1=f(p,\inff).
\]
If $p\geq 0$, since $(\tsum_{i=1}^{n}\tfrac{1}{i})_{n\in\N}$ is not a convergent sequence, there is $u\in\N$ such that $p<\tsum_{i=1}^{u}\tfrac{1}{i}$. Moreover, for each $s>p$, there is some $m\in\N$ such that $s-p>\tfrac{1}{m}$ by the Archimedean property. Set
\[u_s\coloneqq\max\{u,m\}.\]
We have
\[\tfrac{1}{u_s}\leq\tfrac{1}{m}<s-p\,\,\,\textrm{ and }p<\tsum_{i=1}^{u}\tfrac{1}{i}\leq\tsum_{i=1}^{u_s}\tfrac{1}{i}.\]

The function $f_{u_s}$ is a piecewise constant function. Let us denote by $0\eqqcolon t_0<t_1<\cdots<t_n\coloneqq\tsum_{i=1}^{u_s}\tfrac{1}{i}$ the discontinuity points of $f_{u_s}$. Note that $p\geq t_0$ and $p<t_n$. Now, let \[i-1\coloneqq\max\{j\in\{0,\ldots ,n-1\}\mid t_j\leq p\}.\]
The interval $[t_{i-1},t_i)$ is such that $p<t_i< s$. Indeed, if $[p,s)\subseteq[t_{i-1},t_i)$, we have $t_{i-1}\leq p$, $s\leq t_i$ and
\[\tfrac{1}{m}<s-p\leq t_i-t_{i-1}\leq\tfrac{1}{u_s}\leq\tfrac{1}{m},\]
which is a contradiction.
Thus, $p<t_i< s$ and from the formula of $f_{u_s}(p,\inff)$ described in Lemma \ref{DecompositionNonnegativeFunction_aux}, we conclude that
\[f(s,\inff)\leq f(t_i,\inff)= f_{u_s}(p,\inff).\]
As a result, and because $\{u_s\mid s>p, s\in\Q\}\subseteq\N$, we obtain
\[f(p,\inff)=\tbigvee_{s>p}f(s,\inff)\leq\tbigvee_{s>p}f_{u_s}(p,\inff)\leq\tbigvee_{k\in\N}f_k(p,\inff)=(\lim_{k\rightarrow+\infty} f_k)(p,\inff).\qedhere\]
\end{proof}

Since any nonnegative $f\in\eF(L)=\eM(\Cong(L))$ that is measurable on $L$ satisfies the conditions required in the previous theorem, we obtain the following corollary for any nonnegative measurable function.

\begin{Corollary}\label{DecompositionNonnegativeMeasFunction}
Let $f\colon \ereals\to\Cong(L)$ be a nonnegative real-valued function on $L$. If $f$ is measurable on $L$, then we can write
\[f=\lim_{k\rightarrow+\infty}\tsum_{i=1}^{k}\tfrac{1}{i}\cdot\rchi_{\theta_{S_i}}\]
for some complemented $\theta_{S_i}\in\Cong(L)$.
\end{Corollary}

In other words, any nonnegative measurable function $f$ on $L$ can be written as a limit of an increasing sequence $(f_k)_{k\in\N}$ of nonnegative simple functions in $\SM(\Cong(L))$.
However, we point out that in general $f_k\coloneqq\tsum_{i=1}^{k}\tfrac{1}{i}\cdot\rchi_{\theta_{S_i}}$ is not measurable on $L$ because $f_k\in\SM(\Cong(L))\cap\SM(L)$ if and only if the congruences $\theta_{S_1},\ldots,\theta_{S_k}$ are clopen.

\section{Integral of a nonnegative simple function}\label{Integral_Simple_Nonnegative_Section}

We will now address the problem of defining the integral of localic functions. We will do it for localic simple functions. We start with the case of nonnegative simple functions.
\medskip

Recall from \cite{Simpson2012} that  a map $\mu\colon L\to[0,\infty]$ on a join-$\sigma$-complete lattice $L$ is  a {\em  measure} on $L$ if
\begin{enumerate}
\item[(M1)] $\mu(0_{L})=0$;
\item[(M2)] $\forall x,y\in L$,\  $x\leq y\Rightarrow\mu(x)\leq\mu(y)$;
\item[(M3)] $\forall x,y\in L$,\  $\mu(x)+\mu(y)=\mu(x\vee y)+\mu(x\wedge y)$;
\item[(M4)] $\forall (x_i)_{i\in\N}\subseteq L$,\ $\forall i\in\N, \ x_i\leq x_{i+1}\Rightarrow\displaystyle\mu(\tbigvee_{i\in\N}x_i)=\sup_{i\in\N}\mu(x_i)$.
\end{enumerate}
In the literature, these functions are  usually referred to as {\em  $\sigma$-continuous valuations} (\cite{KeimelOthers2003,Simpson2012}).

From now on, let $L$ be a $\sigma$-frame and let $\mu$ be a measure on $\SL(L)$. Recall that the lattice $\Cong(L)$ is dual to $\SL(L)$. For each $S\in \SL(L)$ we will denote the corresponding congruence by $\theta_S$ (hence $S=L/\theta_S$).

\begin{Definition}[\bf{Integral of a nonnegative simple function}]\label{Integral_Simple_nonnegative}
If $g\in\SM(\Cong(L))$ is a nonnegative function with canonical representation
\[g=\tsum_{i=1}^{n}r_i\cdot\rchi_{\theta_{S_i}^{c}},\]
the {\em integral of $g$ with respect to $\mu$} (briefly, {\em $\mu$-integral}) is the value
\[\int_L g\,d\mu\equiv\int g\,d\mu\coloneqq\tsum_{i=1}^{n}r_i\mu(S_i).\]
Moreover, for each $S\in\SL(L)$, the {\em integral of $g$ over $S$ with respect to $\mu$} is given by
\[\int_{S}g\,d\mu\coloneqq\tsum_{i=1}^{n}r_i\mu(S_i\wedge S).\]
\end{Definition}

By convention, whenever $r_i=0$ and $\mu(S_i)=\infty$ (or $\mu(S_i\wedge S)=\infty$) we set $0\,\cdot\,\infty\coloneqq0$.
Because $g$ is simple and nonnegative, we have $0\leq r_1<\cdots<r_n$ (by Corollary \ref{SFunction_canonical}), and since $\mu(S_i\wedge S)\in[0,\infty]$ for every $i=1,\ldots ,n$, not only is the $\mu$-integral well-defined but we also get that
\[0\leq\int_{S}g\,d\mu\leq\infty.\]

We will omit ``with respect to $\mu$'' or drop the prefix $\mu$ if there is no ambiguity.
When a simple function $g$ is not expressed in its canonical form, it will be helpful to evaluate the integral of $g$ through a weaker representation.

\begin{Proposition}\label{NonSimpleProp01}
Let $g=\tsum_{i=1}^{n}r_i\cdot\rchi_{\theta_{S_i}^{c}}$ be a representation of a nonnegative simple function with pairwise disjoint $\theta_{S_1}^{c}, \theta_{S_2}^{c},\ldots ,\theta_{S_n}^{c}$ in $\Cd\Cong(L)$. Then $\mu(S_i\wedge S)=0$ whenever $r_i<0$, and for any $S\in\SL(L)$,
\[\int_{S}g\,d\mu=\tsum_{i=1}^{n}r_i\mu(S_i\wedge S).\]
\end{Proposition}

\begin{Remark}
In other words, if $g=\tsum_{i=1}^{n}r_i\cdot\rchi_{\theta_{S_i}}$, where $\theta_{S_1},\ldots ,\theta_{S_n}$ are pairwise disjoint in $\Cd\Cong(L)$, then
\[\int_{S}g\,d\mu=\tsum_{i=1}^{n}r_i\mu(S_i^{c}\wedge S).\]
\end{Remark}

It is worth noting that the integral is defined such that
\[\int\rchi_{\theta_S}d\mu=\mu(S^{c})\]
because, while it seems more natural to think of $\rchi_{\theta_S}$ as the characteristic function in $\F(L)$ associated with a complemented $S\in\SL(L)$, if this were the case, taking $S_1\leq S_2$ in $\Cd\Cong(L)$, one would expect the characteristic function associated with $S_1$ to be less or equal than the one associated with $S_2$, but $\rchi_{\theta_{S_2}}\leq\rchi_{\theta_{S_1}}$. Thus, we set
\[\rchi_S\coloneqq\rchi_{\theta_S^{c}},\]
and through this identification the integral of a nonnegative simple function with canonical representation $\tsum_{i=1}^{n}r_i\cdot\rchi_{\theta_{S_i}^{c}}$ becomes
\[\int\tsum_{i=1}^{n}r_i\cdot\rchi_{S_i}\,d\mu=\tsum_{i=1}^{n}r_i\mu(S_i).\]

\begin{Example}
Recalling that $0_{\Cong(L)}\in\Cong(L)$ represents $L=1_{\SL(L)}\in\SL(L)$ and $1_{\Cong(L)}\in\Cong(L)$ represents $\{1_L\}=0_{\SL(L)}\in\SL(L)$, the integrals of $\pmb{0}$ and $\pmb{1}$ are given, respectively, by
\begin{gather*}
\int\pmb{0}\,d\mu=\int\rchi_{0_{\Cong(L)}}\,d\mu=\mu(1_{\SL(L)}^{c})=\mu(\{1_L\})=0\,\,\textrm{ and }\\
\int\pmb{1}\,d\mu=\int\rchi_{1_{\Cong(L)}}\,d\mu=\mu(0_{\SL(L)}^{c})=\mu(L).
\end{gather*}
\end{Example}

As expected, the integral of a nonnegative simple function over a $\sigma$-sublocale $S$ is monotone with respect to $S$:

\begin{Proposition}\label{NonSimpleProp02}
If $g\in\SM(\Cong(L))$ is nonnegative and $S,T\in\SL(L)$ are such that $S\leq T$, then
\[\int_{S}g\,d\mu\leq\int_{T}g\,d\mu.\]
\end{Proposition}

\begin{proof}
Write $g=\tsum_{i=1}^{n}r_i\cdot\rchi_{\theta_{S_i}}$ where $\theta_{S_1}, \theta_{S_2},\ldots ,\theta_{S_n}$ are pairwise disjoint in $\Cd\Cong(L)$. Then
\[\int_{S}g\,d\mu=\tsum_{i=1}^{n}r_i\mu(S_i^{c}\wedge S)\leq\tsum_{i=1}^{n}r_i\mu(S_i^{c}\wedge T)=\int_{T}g\,d\mu.\qedhere\]
\end{proof}

Furthermore, whenever $S\in\SL(L)$ is complemented, the integral of $g$ over $S$ is equal to the integral of $g\cdot\rchi_S$.

\begin{Proposition}\label{NonSimpleProp03}
Given a nonnegative $g\in\SM(\Cong(L))$ and a complemented $S\in\SL(L)$,
\[\int_{S}g\,d\mu=\int g\cdot\rchi_{\theta_{S}^{c}}\,d\mu.\]
\end{Proposition}

\begin{proof}
Write $g=\tsum_{i=1}^{n}r_i\cdot\rchi_{\theta_{S_i}^{c}}$ with $\theta_{S_1}^{c},\ldots ,\theta_{S_n}^{c}\in\Cd\Cong(L)$ pairwise disjoint. Since
$g\cdot\rchi_{\theta_S^{c}}=\tsum_{i=1}^{n}r_i\cdot\rchi_{\theta_{S_i}^{c}\wedge\theta_S^{c}}$, where $\theta_{S_1}^{c}\wedge\theta_S^{c},\ldots ,\theta_{S_n}^{c}\wedge\theta_S^{c}$ are also pairwise disjoint, and $\theta_{S_i}^{c}\wedge\theta_S^{c}=\theta_{{S_i}^{c}\vee S^{c}}$ for $i=1,\ldots,n$, we have
\[
\int g\cdot\rchi_{\theta_S^{c}}\,d\mu=\tsum_{i=1}^{n}r_i\mu((S_i^{c}\vee S^{c})^{c})=\tsum_{i=1}^{n}r_i\mu(S_i\wedge S)=\int_{S}g\,d\mu.\qedhere
\]
\end{proof}

\section{Integral of a general simple function}\label{Integral_Simple_Section}

After studying the nonnegative case, let us consider a general simple function $g\in\F(L)$. Recall that $g=g^{+}-g^{-}$, where $g^{+}$ and $g^{-}$ are both nonnegative simple functions. Thus, the idea is to define the integral of $g$ through the integrals of $g^{+}$ and $g^{-}$.

\begin{Definition} [\bf{Integral of a simple function}] \label{Integral_Simple}
Let $g\in\SM(\Cong(L))$ and $S\in\SL(L)$. If
\[\int_{S}g^{+}\,d\mu<\infty\qquad\textrm{ or }\qquad\int_{S}g^{-}\,d\mu<\infty,\]
we say that $g$ is {\em $\mu$-integrable over $S$}, and the {\em $\mu$-integral of $g$ over $S$} is
\[\int_{S}g\,d\mu\coloneqq\int_{S}g^{+}\,d\mu-\int_{S}g^{-}\,d\mu.\]
We say that $g$ is {\em $\mu$-integrable} if $g$ is $\mu$-integrable over $L=1_{\SL(L)}$, and in that case we talk about the {\em $\mu$-integral} of $g$.
\end{Definition}

Note that a nonnegative simple function is always $\mu$-integrable over any $S\in\SL(L)$. When there is no ambiguity, we will drop the prefix $\mu$.
Moreover, the previous definition generalises Definition \ref{Integral_Simple_nonnegative} in the sense that whenever $g$ is nonnegative
\[
\int_S^{(\textrm{Def. }\ref{Integral_Simple})} g\, d\mu=\int_S^{(\textrm{Def. }\ref{Integral_Simple_nonnegative})} g\, d\mu
\]
because $g=g^{+}$ and $g^{-}=\pmb{0}$.

\begin{Definition}\label{Summable}
A $g\in\SM(\Cong(L))$ is {\em summable over $S\in\SL(L)$} if
\[
\int_{S} g^{+}\,d\mu<\infty\qquad\textrm{ and }\qquad\int_{S} g^{-}\,d\mu<\infty.
\]
We say that $g$ is {\em summable} if $g$ is summable over $L=1_{\SL(L)}$.
\end{Definition}

By definition, it is clear that every summable function is integrable and its integral is finite.

\begin{Remark}
Here, our use of the term ``integrable'' admits that the integral of $g$ can be equal to $+\infty$ or $-\infty$. However, the standard approach in measure theory is that a function is integrable whenever it is summable.
\end{Remark}

The following formulas will be often helpful to determine the integral of a integrable simple function.

\begin{Proposition}\label{SimpleProp01}
If $g\in\SM(\Cong(L))$ is integrable over $S\in\Cong(L)$ and $g=\tsum_{i=1}^{n}r_i\cdot\rchi_{\theta_{S_i}^{c}}$ is a representation of $g$ with $\theta_{S_1}^{c},\ldots ,\theta_{S_n}^{c}$ pairwise disjoint in $\Cd\Cong(L)$, then
\[\int_{S} g\,d\mu=\tsum_{i=1}^{n}r_i\mu(S_i\wedge S).\]
\end{Proposition}

\begin{proof}
If $r_i\geq0$ for all $i\in\{1,\ldots ,n\}$, then $g$ is nonnegative and the result follows from Proposition \ref{NonSimpleProp01}.
On the other hand, if there is some $i\in\{1,\ldots ,n\}$ such that $r_i<0$, let  $r_1\leq\cdots \leq r_n$ (with no loss of generality) and set $k\coloneqq\max\{i\mid r_i<0\}$. Then
$g^{-}=\tsum_{i=1}^{k}(-r_i)\cdot\rchi_{\theta_{S_i}^{c}}$
and $g^{+}=\tsum_{i=k+1}^{n}r_i\cdot\rchi_{\theta_{S_i}^{c}}$.
Hence, once again by Proposition \ref{NonSimpleProp01},
\[
\int_{S}g^{+}\,d\mu-\int_{S}g^{-}\,d\mu=\tsum_{i=k+1}^{n}r_i\mu(S_i\wedge S)-\tsum_{i=1}^{k}(-r_i)\mu(S_i\wedge S)=\tsum_{i=1}^{n}r_i\mu(S_i\wedge S).\]
\end{proof}

\begin{Proposition}\label{SimpleProp02}
Let $g\in\SM(\Cong(L))$ and let $S\in\SL(L)$ be complemented. If $g$ is integrable over $S$, then $g\cdot\rchi_{\theta_S^{c}}$ is integrable and
\[\int_{S}g\,d\mu=\int g\cdot\rchi_{\theta_S^{c}}\,d\mu.\]
\end{Proposition}

\begin{proof}
If $g$ is nonnegative, the result follows from Proposition \ref{NonSimpleProp03}. In the case of a general simple function $g$, note that $(g\cdot\rchi_{\theta_S^{c}})^{+}=g^{+}\cdot\rchi_{\theta_S^{c}}$ and $(g\cdot\rchi_{\theta_S^{c}})^{-}=g^{-}\cdot\rchi_{\theta_S^{c}}$. Thus, because $(g\cdot\rchi_{\theta_S^{c}})^{+}$ and $(g\cdot\rchi_{\theta_S^{c}})^{-}$ are nonnegative,
\begin{gather*}
\int(g\cdot\rchi_{\theta_S^{c}})^{+}d\mu=\int g^{+}\cdot\rchi_{\theta_S^{c}}\,d\mu=\int_S g^{+}\,d\mu\,\,\textrm{ and }\\
\int(g\cdot\rchi_{\theta_S^{c}})^{-}d\mu=\int g^{-}\cdot\rchi_{\theta_S^{c}}\,d\mu=\int_S g^{-}\,d\mu.\qedhere
\end{gather*}
\end{proof}

We have seen that the integral of a nonnegative simple function over an $S\in\SL(L)$  is always nonnegative. In the following, we will weaken the assumption that ``a simple function $g$ is nonnegative'' so that the integral of $g$ over complemented $\sigma$-sublocales is still nonnegative.

\begin{Proposition}\label{SimpleProp03}
If $g\in\SM(\Cong(L))$ is integrable over a complemented $\sigma$-sublocale $S\in\SL(L)$ and $\theta_S^{c}\wedge g(\inff,0)=0$, then
\[
\int_{S}g\,d\mu\geq 0.
\]
\end{Proposition}

\begin{proof}
Write $g=\tsum_{i=1}^{n}r_i\cdot\rchi_{\theta_{S_i}^{c}}$ where $\theta_{S_1}^{c},\ldots ,\theta_{S_n}^{c}$ are pairwise disjoint in $\Cd\Cong(L)$. By Proposition \ref{SimpleProp01},
\[\int_{S} g\,d\mu=\tsum_{i=1}^{n}r_i\mu(S_i\wedge S).\]

If $r_i\geq 0$ for  $i=1,\ldots ,n$, then the result is trivial. On the other hand, if there is some $i\in\{1,\ldots ,n\}$ such that $r_i< 0$, let $k\coloneqq\max\{i\mid r_i< 0\}$ and
take $q\in\Q$ such that $r_i\leq r_k< q< 0$. Then $\theta_S^{c}\wedge g(\inff,q)\leq \theta_S^{c}\wedge g(\inff,0)=0$, and by Proposition \ref{Other_formula} we get
\[\theta_S^{c}\wedge\tbigwedge_{j=i+1}^{n}\theta_{S_j}\leq\theta_S^{c}\wedge\tbigwedge_{j=k+1}^{n}\theta_{S_j}=\theta_S^{c}\wedge\big(\tbigvee_{j=k+1}^{n}\theta_{S_j}^{c}\big)^{c}=\theta_{S}^{c}\wedge g(\inff,q)=0.\]
Finally, as $\theta_{S_i}\vee\theta_{S_j}=1$ for any $j\neq i$, we have $\theta_{S_i}\vee\tbigwedge_{j=i+1}^{n}\theta_{S_j}=1$ and
\[\theta_{S_i}\vee\theta_S^{c}=(\theta_{S_i}\vee\theta_S^{c})\wedge\big(\theta_{S_i}\vee\tbigwedge_{j=i+1}^{n}\theta_{S_j}\big)=\theta_{S_i}\vee\big(\theta_S^{c}\wedge\tbigwedge_{j=i+1}^{n}\theta_{S_j}\big)=\theta_{S_i}.\]
Hence $S_i\leq S^{c}$ whenever $r_i< 0$, which implies $\mu(S\wedge S_i)=\mu(0_{\SL(L)})=0$.
\end{proof}

Next, we present some elementary properties of the integral.

\begin{Proposition}\label{SimpleProp04}
Let $g\in\SM(\Cong(L))$ and $\lambda\in\Q$. If $g$ is integrable over $S\in\SL(L)$, then $\lambda\cdot g$ is also integrable over $S$ and
\[\int_{S}\lambda\cdot g\,d\mu=\lambda\int_{S}g\,d\mu.\]
\end{Proposition}

\begin{proof}
Write $g=\tsum_{i=1}^{n}r_i\cdot\rchi_{\theta_{S_i}^{c}}$ with $\theta_{S_1}^{c},\ldots ,\theta_{S_n}^{c}\in\Cd\Cong(L)$ pairwise disjoint. Recall that
\[\lambda\cdot g=\tsum_{i=1}^{n}(\lambda r_i)\cdot\rchi_{\theta_{S_i}^{c}}.\]
If $\lambda=0$, then $\lambda\cdot g=\pmb{0}$ and the claim holds trivially.

Suppose that $\lambda>0$. If $g$ is nonnegative, then
\[
\int_{S}\lambda\cdot g\,d\mu=\tsum_{i=1}^{n}\lambda r_i\mu(S_i\wedge S)=\lambda\tsum_{i=1}^{n}r_i\mu(S_i\wedge S)=\lambda\int_{S}g\,d\mu.
\]
If $g$ is a general simple function, since $(\lambda\cdot g)^{+}=\lambda\cdot g^{+}$ and $(\lambda\cdot g)^{-}=\lambda\cdot g^{-}$, we have
\[\int_{S}(\lambda\cdot g)^{+}\,d\mu=\lambda\int_{S} g^{+}\,d\mu\,\,\textrm{ and }\,\,\int_{S}(\lambda\cdot g)^{-}\,d\mu=\lambda\int_{S} g^{-}\,d\mu.\]
So $g$ is integrable over $S$ if and only if $\lambda\cdot g$ is integrable over $S$, and
\[
\int_{S}\lambda\cdot g\,d\mu=\lambda\int_{S} g^{+}\,d\mu-\lambda\int_{S} g^{-}\,d\mu=\lambda\big(\int_{S} g^{+}\,d\mu-\int_{S} g^{-}\,d\mu\big)=\lambda\int_{S} g\,d\mu.
\]

Finally, for $\lambda<0$  the arguments are similar to the ones used for a positive scalar and a general simple function, taking note that $\lambda\cdot g=(-\lambda)\cdot(-g)$ with $-\lambda>0$, $(\lambda\cdot g)^{+}=(-\lambda)\cdot g^{-}$ and $(\lambda\cdot g)^{-}=(-\lambda)\cdot g^{+}$.
\end{proof}

\begin{Proposition}\label{SimpleProp05}
Let $g,h\in\SM(\Cong(L))$ be integrable over $S\in\SL(L)$. If $g+h$ is also integrable over $S$, then
\[\int_{S}g\,d\mu+\int_{S}h\,d\mu=\int_{S}(g+h)\,d\mu.\]
\end{Proposition}

\begin{proof}
Recalling Proposition \ref{Prop_SimpleFunctions} and the expression of $g+h$ obtained in its proof, this is a straightforward consequence of $\mu$ being a measure.
\end{proof}

In particular, we can apply Proposition \ref{SimpleProp05} whenever $g+h$ is nonnegative for some $g,h\in\SM(\Cong(L))$. Combining Propositions \ref{SimpleProp05} and \ref{SimpleProp04}, we get the monotonicity of the integral.

\begin{Proposition}\label{SimpleProp07}
Let $g,h\in\SM(\Cong(L))$ with $g\leq h$. If $g$ and $h$ are integrable over $S\in\SL(L)$, then
\[\int_{S}g\,d\mu\leq\int_{S}h\,d\mu.\]
\end{Proposition}

\begin{proof} As $h-g$ is nonnegative, and therefore integrable over $S$,
\[0\leq\int_S h-g\,d\mu=\int_S h\,d\mu-\int_S g\,d\mu.\qedhere\]
\end{proof}

A more technical proof allows us to weaken the assumption that $g\leq h$ when we are taking the integral over a complemented $\sigma$-sublocale $S$. We point out that $h-g$ does not need to be integrable over $S$ in the next proposition.

\begin{Proposition}
If $g,h\in\SM(\Cong(L))$ are integrable over a complemented $S\in\SL(L)$ such that $\theta_S^{c}\wedge(h-g)(\inff,0)=0$, then
\[\int_{S}g\,d\mu\leq\int_{S}h\,d\mu.\]
\end{Proposition}

\begin{proof}
Write $g=\tsum_{i=1}^{n}r_i\cdot\rchi_{\theta_{S_i}^{c}}$ and $h=\tsum_{j=1}^{m}s_j\cdot\rchi_{\theta_{T_j}^{c}}$ in their canonical representations. Recalling the proof of Proposition \ref{Prop_SimpleFunctions},
\begin{gather}
h-g=\tsum_{j=1}^{m}\tsum_{i=1}^{n}(s_j-r_i)\cdot\rchi_{\theta_{S_i}^{c}\wedge\theta_{T_j}^{c}},\notag\\
\qquad g=\tsum_{i=1}^{n}\tsum_{j=1}^{m}r_i\cdot\rchi_{\theta_{S_i}^{c}\wedge\theta_{T_j}^{c}}\,\,\textrm{ and }\,\,h=\tsum_{j=1}^{m}\tsum_{i=1}^{n}s_j\cdot\rchi_{\theta_{S_i}^{c}\wedge\theta_{T_j}^{c}},\tag{$*$}\label{ast}
\end{gather}
where $\theta_{S_i}^{c}\wedge\theta_{T_j}^{c}$ are pairwise disjoint elements with $\tbigvee_{i=1}^{n}\tbigvee_{j=1}^{m}(\theta_{S_i}^{c}\wedge\theta_{T_j}^{c})=1$.

For the previous representation of $h-g$ to be canonical, all that is left is to ensure that the coefficients are sorted in strict ascending order and to remove the terms $\theta_{S_i}^{c}\wedge\theta_{T_j}^{c}=0$. After taking those steps (recall Proposition \ref{FSimple_equiv}), the canonical representation of $h-g$ will be a finite sum of the form
\[
h-g=\tsum a_{l}\cdot\rchi_{\theta_{R_l}^{c}},
\]
where $a_l=s_j-r_i$ for some pair $(i,j)$ and $\theta_{R_l}^{c}\geq\theta_{S_i}^{c}\wedge\theta_{T_j}^{c}$.

If $a_l\geq 0$ for all $l$, then $g\leq h$ and the claimed follows from the previous proposition. If there is $l$ such that $a_l=s_j-r_i< 0$, set $k\coloneqq\max\{l\mid a_l< 0\}$.
Since $\Q$ is dense in $\R$, there is $q\in\Q$ such that $s_j-r_i=a_l\leq a_k< q< 0$. By Corollary \ref{SFunction_canonical}, $(h-g)(\inff,q)=\tbigvee_{n=1}^{k}\theta_{R_n}^{c}\geq \theta_{R_l}^{c}\geq\theta_{S_i}^{c}\wedge\theta_{T_j}^{c}$.
Thus,
\[
\theta_S^{c}\wedge\theta_{S_i}^{c}\wedge\theta_{T_j}^{c} \leq\theta_S^{c}\wedge (h-g)(\inff,q)\leq \theta_{S}^{c}\wedge (h-g)(\inff,0)=0.
\]
This shows that $s_j-r_i<0$ implies that $S^{c}\vee S_i^{c}\vee T_j^{c}=1$. In other words, $S\wedge S_i\wedge T_j\neq0$ implies that $s_j\geq r_i$. Hence, representations \eqref{ast} of $g$ and $h$ and Proposition \ref{SimpleProp01} yield the desired inequality.
\end{proof}

Finally, it follows from Propositions \ref{SimpleProp05} and \ref{SimpleProp07} that

\begin{Proposition}\label{SimpleProp06}
If $g,h\in\SM(\Cong(L))$ are summable over $S\in\SL(L)$, then $g+h$ is also summable over $S$.
\end{Proposition}

\begin{proof}
Set $f\coloneqq g+h$. Then $f^{+}-f^{-}=g^{+}-g^{-}+h^{+}-h^{-}$ and
\[
f^{+}+g^{-}+h^{-}=g^{+}+h^{+}+f^{-}.
\]
As $g$ and $h$ are summable over $S$, $f^{+}=(g+h)\vee\,\pmb{0}\leq g^{+}+h^{+}$ and $f^{-}\leq f^{+}+g^{-}+h^{-}$, it follows from fact that $g^{+}+h^{+}$ and $f^{+}+g^{-}+h^{-}$ are nonnegative, and thus integrable over $S$, that $f$ is summable over $S$:
\begin{align*}
\int_S f^{+}\,d\mu&\leq \int_S g^{+}\,d\mu+\int_S h^{+}\,d\mu<\infty\textrm{ and }\\
\int_S f^{-}\,d\mu&\leq \int_S f^{+}\,d\mu+\int_S g^{-}\,d\mu+\int_S h^{-}\,d\mu<\infty.\qedhere
\end{align*}
\end{proof}

\medskip
As a result, the integral is linear on the class of summable simple functions, in the sense that for any $r,s\in\Q$ and any $g,h\in\SM(\Cong(L))$ summable over $S\in\SL(L)$,
\[\int_S (r\cdot g+s\cdot h)\,d\mu=r\int_S g\,d\mu\,+\,s\int_S h\,\mu.\]

\section{The indefinite integral}

Given a simple function $g\in\F(L)$, the map $\eta\colon \SL(L)\to[0,\infty]$ defined by
\[
\eta(S)\coloneqq\int_{S}g\,d\mu
\]
is called the {\em indefinite integral} of $g$. In this section, we want to show that whenever $g$ is nonnegative, $\eta$ is a measure on $\SL(L)$.

\begin{Lemma}\label{IntSF_Modular}
For any integrable $g\in\SM(\Cong(L))$ and any $S, T\in\SL(L)$,
\[\int_{S}g\,d\mu+\int_{T}g\,d\mu=\int_{S\vee T}g\,d\mu+\int_{S\wedge T}g\,d\mu.\]
\end{Lemma}

\begin{proof}
First, note that if $g$ is integrable, then $g$ is integrable over any $\sigma$-sublocale of $L$. Now, write $g=\tsum_{i=1}^{n}r_i\cdot\rchi_{\theta_{S_i}^{c}}$ with $\theta_{S_1}^{c},\ldots ,\theta_{S_n}^{c}\in\Cd\Cong(L)$ pairwise disjoint. Since $\mu$ is modular (M3), we have
\[
\begin{split}
\int_{S}g\,d\mu+\int_{T}g\,d\mu&=\tsum_{i=1}^{n}r_i\mu(S_i\wedge S)+\tsum_{i=1}^{n}r_i\mu(S_i\wedge T)\\
&=\tsum_{i=1}^{n}r_i[\mu(S_i\wedge S)+\mu(S_i\wedge T)]\\
&=\tsum_{i=1}^{n}r_i[\mu(S_i\wedge(S\wedge T))+\mu(S_i\wedge(S\vee T))]\\
&=\tsum_{i=1}^{n}r_i\mu(S_i\wedge(S\wedge T))+\tsum_{i=1}^{n}r_i\mu(S_i\wedge(S\vee T))\\
&=\int_{S\wedge T}g\,d\mu+\int_{S\vee T}g\,d\mu.\qedhere
\end{split}
\]
\end{proof}

\begin{Lemma}\label{IntSF_SigmaContinuous}
Let $g\in\SM(\Cong(L))$ be integrable. If $(B_k)_{k\in\N}$ is an increasing sequence in $\SL(L)$ and $B=\tbigvee_{k\in\N}B_k$, then
\[\int_{B}g\,d\mu=\lim_{k\rightarrow+\infty}\int_{B_k}g\,d\mu.\]
In particular, if $g$ is nonnegative,
\[\int_{B}g\,d\mu=\sup_{k\in\N}\int_{B_k}g\,d\mu.\]
\end{Lemma}

\begin{proof}
Write $g=\tsum_{i=1}^{n}r_i\cdot\rchi_{\theta_{S_i}^{c}}$ with $\theta_{S_1}^{c},\ldots ,\theta_{S_n}^{c}\in\Cd\Cong(L)$ pairwise disjoint. For each $k\in\N$, $g$ is integrable over $B_k$ and
\[\int_{B_k}g\,d\mu=\tsum_{i=1}^{n}r_i\mu(S_i\wedge B_k).\]
Moreover, $(S_i\wedge B_k)_{k\in\N}$ is increasing, $\mu$ is $\sigma$-continuous (M4), and each $S_i$ being complemented implies that
\[\tbigvee_{k\in\N}(S_i\wedge B_k)= S_i\wedge\tbigvee_{k\in\N}B_k=S_i\wedge B.\]
Therefore,
\[
\begin{split}
\lim_{k\rightarrow+\infty}\int_{B_k}g\,d\mu&=\lim_{k\rightarrow+\infty}\tsum_{i=1}^{n}r_i\mu(S_i\wedge B_k)=\tsum_{i=1}^{n}r_i\lim_{k\rightarrow+\infty}\mu(S_i\wedge B_k)\\
&=\tsum_{i=1}^{n}r_i\sup_{k\in\N}\mu(S_i\wedge B_k)=\tsum_{i=1}^{n}r_i\mu\big(\tbigvee_{k\in\N}(S_i\wedge B_k)\big)\\
&=\tsum_{i=1}^{n}r_i\mu(S_i\wedge B)=\int_{B}g\,d\mu.\qedhere
\end{split}
\]
\end{proof}

\begin{Theorem}
The indefinite integral of a nonnegative function is a measure on $\SL(L)$.
\end{Theorem}

\begin{proof}
Let $g$ be a nonnegative simple function and $\eta$ its indefinite integral.
By the definition of the integral over a $\sigma$-sublocale, (M1) holds trivially. Further, (M2), (M3) and (M4) follow from Proposition \ref{NonSimpleProp02}, Lemma \ref{IntSF_Modular} and Lemma \ref{IntSF_SigmaContinuous}, respectively.
\end{proof}

\section{Point-free setting versus classic setting}\label{Classic_vs_Pointfree_Section}

In this section we show that our definition of integral for simple functions extends the classic Lebesgue integral.
\medskip

Let $(X,\mathcal{A})$ be a {\em measurable space}, i.e., a set $X$ equipped with a $\sigma$-algebra $\mathcal{A}\subseteq\powerset$.
Consider a simple function $\tilde{f}\colon X\to\R$ with codomain in $\Q$, that is, a function of the form
\[\tilde{f}=\tsum_{i=1}^{n}r_i\mathds{1}_{A_i},\]
for some $n\in\N$ and $r_1,\ldots ,r_n\in\Q$, where $\mathds{1}_{A_i}\colon X\to\{0,1\}$ is the indicator (characteristic) function of $A_i$ and $A_1,\ldots ,A_n\in\mathcal{A}$ are pairwise disjoint subsets of $X$ such that  $\tbigcup_{i=1}^{n}A_i=X$. With no loss of generality, we may suppose that $r_1< r_2<\cdots < r_n$. The following table lists the values of $\tilde{f}^{-1}(]-\infty,r[)$ and
$\tilde{f}^{-1}(]r, +\infty[)$ for each $r\in\Q$:

\medskip
\begin{center}
\begin{tabular}{cc|cc}
\toprule[1.25pt] \addlinespace[0.5em]
\multicolumn{2}{c}{$\tilde{f}^{-1}(]-\infty,r[)$} &  \multicolumn{2}{c}{$\tilde{f}^{-1}(]r, +\infty[)$}\\[2mm]
\toprule[1.25pt] \addlinespace[0.5em]
      $\emptyset$ &  if $r\leq r_1$ & $X$ & if $r< r_1$ \\[2mm]
      $A_1$ &  if $r_1<r\leq r_2$ &  $\displaystyle\tbigcup_{i=2}^{n}A_i$ & if $r_1\leq r< r_2$ \\[4mm]
     $A_1\cup A_2$ & if $r_2<r\leq r_3$ &  $\displaystyle\tbigcup_{i=3}^{n}A_i$ &  if $r_2\leq r<r_3$ \\[4mm]
      $\vdots$ & $\vdots$ &      $\vdots$ & $\vdots$\\[2mm]
      $\displaystyle\tbigcup_{i=1}^{n-1}A_i$ &  if $r_{n-1}<r\leq r_n$ &$ A_n$ &  if $r_{n-1}\leq r<r_n$ \\[4mm]
      $X$ & if $r>r_n$ &    $\emptyset$ & if $r\geq r_n$\\[2mm]
      \bottomrule[1.25pt]
\end{tabular}
\end{center}

\medskip
Now consider the $\sigma$-frame $L=\mathcal{A}$. The localic counterpart of the classic simple function $\tilde{f}$ is the localic function
\[f\colon \reals\to\mathcal{A}\in \M(\mathcal{A})\]
determined by the formulas $f(\inff,q)\coloneqq\tilde{f}^{-1}(]-\infty,q[)$ and $f(p,\inff)\coloneqq\tilde{f}^{-1}(]p,+\infty[)$ for all $p,q\in\Q$. In other words, $f$ is the localic simple function
\[f=\tsum_{i=1}^{n}r_i\cdot\rchi_{A_i}.\]
Through the isomorphism $\nabla:\mathcal{A}\to\nabla[\mathcal{A}]$ that embeds $\mathcal{A}$ in $\Cong(L)$, we can identify $f\in\M(\mathcal{A})$ with 
\[\nabla\circ f=\tsum_{i=1}^{n}r_i\cdot\rchi_{\nabla_{A_i}}\in\F(\mathcal{A}),\]
and regard $f$ as an element of $\F(\mathcal{A})$.

This shows that every simple function $\tilde{f}\colon X\to\R$ with codomain in $\Q$ has a localic counterpart in $\M(\mathcal{A})\subseteq\F(\mathcal{A})$. It is then easy to check that simple functions $\tilde{f}\colon X\to\R$ with codomain in $\Q$ are in a one-to-one correspondence  with the simple functions $f\colon \reals\to\Cong(\mathcal{A})$ that are measurable on $\mathcal{A}$. Note that, however, there may exist simple functions in $\F(\mathcal{A})$ which are not in $\M(\mathcal{A})$.

 Next, let
$\lambda\colon \mathcal{A}\to[0,\infty]$ be a measure on $(X,\mathcal{A})$.
By definition, $\lambda$ is a $\sigma$-additive map such that $\lambda(\emptyset)=0$. As $\mathcal{A}$ is a $\sigma$-algebra, $\lambda$ satisfies the axioms (M1)-(M4) (\cite{Simpson2012}), so it is a measure on the join-$\sigma$-complete lattice
\[L=\mathcal{A}\cong\mathfrak{o}[\mathcal{A}]\subseteq\SL(\mathcal{A}).\]

\begin{Remark}
Since $\lambda$ is a measure on $\mathcal{A}$ and every element of $\mathcal{A}$ is complemented, we can apply Theorem 1 of \cite{Simpson2012} to obtain a measure
\[\lambda^{\diamond}\colon \SL(\mathcal{A})\to[0,\infty]\]
on $\SL(\mathcal{A})$ extending $\lambda$, in the sense that $\lambda^{\diamond}(\mathfrak{o}(A))=\lambda(A)$ for every $A\in\mathcal{A}$.
\end{Remark}

Suppose that $\tilde{f}$ is nonnegative, i.e., $0<r_1<\cdots<r_n$. Recall that the (Lebesgue) $\lambda$-integral of $\tilde{f}$ is the value
\[
\int_{X}\tilde{f}\,d\lambda=\tsum_{i=1}^{n}r_i\lambda(A_i).
\]
Then the Lebesgue $\lambda$-integral of $\tilde{f}$ is equal to the localic $\lambda^{\diamond}$-integral of $f$:
\[
\begin{split}
\int_{\mathcal{A}} f\,d\lambda^{\diamond}&=\int_{\mathcal{A}}\,\tsum_{i=1}^{n}r_i\cdot\rchi_{\nabla_{A_i}}d\lambda^{\diamond}
=\tsum_{i=1}^{n}r_i\lambda^{\diamond}(\mathfrak{c}(A_i)^{c})\\
&=\tsum_{i=1}^{n}r_i\lambda^{\diamond}(\mathfrak{o}(A_i))=\tsum_{i=1}^{n}r_i\lambda(A_i)=\int_X\tilde{f}\,d\lambda.
\end{split}
\]

If $\tilde{f}$ is a general simple function, we say that $\tilde{f}$ is $\lambda$-integrable (\cite{EvansGariepy2015}) if
\[\int_X\tilde{f}^{+}\,d\lambda<\infty\,\,\textrm{ or }\,\,\int_X\tilde{f}^{-}\,d\lambda<\infty.\]
Since $\tilde{f}^{+}$ and $\tilde{f}^{-}$ are nonnegative simple functions whose localic counterparts are $f^{+}$ and $f^{-}$, respectively, we have
\[\int_X\tilde{f}^{+}\,d\lambda=\int_{\mathcal{A}} f^{+}\,d\lambda^{\diamond}\,\,\textrm{ and }\,\,\int_X\tilde{f}^{-}\,d\lambda=\int_{\mathcal{A}} f^{-}\,d\lambda^{\diamond}.\]
Thus, not only $\tilde{f}$ is $\lambda$-integrable if and only if $f$ is $\lambda^{\diamond}$-integrable, but we also get that
\[\int_X\tilde{f}\,d\lambda=\int_X\tilde{f}^{+}\,d\lambda-\int_X\tilde{f}^{-}\,d\lambda=\int_{\mathcal{A}} f^{+}\,d\lambda^{\diamond}-\int_{\mathcal{A}} f^{-}\,d\lambda^{\diamond}=\int_{\mathcal{A}} f\,d\lambda^{\diamond}\]
whenever they are integrable.

\section{Integral of more general functions}

We close with a brief comment on the possibility of extending our definition of integral to a class of localic functions broader than the simple ones.

Let $L$ be a $\sigma$-frame and let $\mu$ be a measure on $\SL(L)$.
In Corollary \ref{DecompositionNonnegativeMeasFunction}, we proved that if $f\in\eF(L)$ is a nonnegative function measurable on $L$, then
\[f=\lim_{n\rightarrow+\infty} g_n=\sup_{n\in\N} g_n,\]
for some increasing sequence $(g_n)_{n\in\N}$ in $\SM(\Cong(L))$ with $\pmb{0}\leq g_n\leq f$ for all $n\in\N$. This suggests that just as we can approximate $f$ by a sequence of simple functions, we can also approximate the integral of $f$ by the integral of simple functions. This motivates the following definition.

\begin{Definition}[\bf{Integral of a nonnegative function}]\label{Integral_nonnegative}
Given a nonnegative $f\in\eF(L)$, the {\em $\mu$-integral of $f$ over $S\in\SL(L)$} is given by
\[
\int_{S} f\,d\mu\coloneqq\sup\Bigl\{\int_{S} g\,d\mu\mid \pmb{0}\leq g\leq f, \,g\in\SM(\Cong(L))\Bigl\}.
\]
The $\mu$-integral of $f$ over $L= 1_{\SL(L)}$ is called the {\em $\mu$-integral of $f$} and we write
\[
\int f\,d\mu\coloneqq\int_{L} f\,d\mu=\sup\Bigl\{\int g\,d\mu\mid \pmb{0}\leq g\leq f, \,g\in\SM(\Cong(L))\Bigl\}.
\]
\end{Definition}

As $f=f^{+}-f^{-}$ for any  $f\in\eF(L)$, we can extend the definition of integral of a nonnegative function to a general function (similarly as in Section \ref{Integral_Simple_Section}).

\begin{Definition}[\bf{Integral of a general function}]\label{Integral}
A function $f\in\eF(L)$ is {\em $\mu$-integrable over $S\in\SL(L)$} if
\[
\int_{S} f^{+}\,d\mu<\infty\qquad\textrm{ or }\qquad\int_{S} f^{-}\,d\mu<\infty,
\]
and its {\em $\mu$-integral over $S$} is given by
\[\int_{S} f\,d\mu\coloneqq\int_{S} f^{+}\,d\mu-\int_{S} f^{-}\,d\mu.\]
The $\mu$-integral of $f$ over $L= 1_{\SL(L)}$ is called the {\em $\mu$-integral of $f$}.
\end{Definition}

It is straightforward to see that Definition \ref{Integral} extends Definition \ref{Integral_nonnegative}, and they both extend Definitions \ref{Integral_Simple_nonnegative} and \ref{Integral_Simple}.
A detailed study of these integrals is left out to a subsequent paper \cite{Bernardes_inpreparation}.

Here, we only point out that some of the elementary properties of the integral of an $f\in\eF(L)$ may not look as neat as in the case of an $f\in\SM(\Cong(L))$. For instance, the indefinite integral of a nonnegative $f\in\eF(L)$ might not necessarily be a measure on $\SL(L)$ because the modularity (axiom (M3)) may fail. Even so, this definition generalises the classic Lebesgue integral in the sense that given a measurable function $\tilde{f}\colon X\to\eR$ on a measure space $(X,\mathcal{A},\lambda)$, the $\lambda$-integral of $\tilde{f}$ is equal to the $\lambda^{\diamond}$-integral of the localic counterpart of $\tilde{f}$ for any measure $\lambda^{\diamond}$ on $\SL(L)$ extending $\lambda$.

\begin{Remark}
Since we work with the frame of reals, which is defined in terms of the ordered field of rationals, we defined a localic simple function using rational scalars (and not real scalars as it is standard in the classical theory). However, this leads to no loss of generality in the proposed point-free integration theory, as we show in \cite{Bernardes_inpreparation}.
\end{Remark}

\subsection*{Acknowledgments} Research partially supported by the Centre for Mathematics of the University of Coimbra (funded
by the Portuguese Government through FCT/MCTES, DOI: 10.54499/UIDB/00324/2020) and by
a PhD grant from FCT/MCTES (PD/BD/11883/2022). The author gratefully acknowledges her PhD
supervisor, Prof. Jorge Picado, for his advice and all suggestions that improved the presentation
of this paper.

\bibliographystyle{amsplain}
\bibliography{Lebesgue_integration_I.bib}

\end{document}